\documentclass[final,onefignum,onetabnum]{siamart220329}

\newcommand{\sR}{\mathbb{R}}
\newcommand{\sH}{\mathcal{H}}

\newcommand{\Fix}{\mathsf{Fix}}
\newcommand{\sol}{\Fix T}
\newcommand{\zer}{\textrm{Zer}}
\newcommand{\spa}{\textrm{span}}

\newcommand{\mysum}{\displaystyle \sum\limits}

\newcommand{\proj}{\mathrm{Proj}}
\newcommand{\Id}{\mathrm{Id}}

\newcommand{\BP}{Banach-Picard }
\newcommand{\KM}{Krasnosel'ski\u{\i}-Mann }
\newcommand{\DR}{Douglas-Rachford }
\newcommand{\FB}{forward-backward }

\newcommand{\E}{\mathcal{E}}

\newcommand{\vj}{\mathbf{1}}
\newcommand{\vo}{\mathbf{0}}

\newcommand{\mI}{\mathbb{I}}
\newcommand{\mO}{\mathbb{O}}

\usepackage{amsfonts}
\usepackage{booktabs}
\usepackage{graphicx}
\usepackage{epstopdf}
\usepackage{algorithmic}
\ifpdf
\DeclareGraphicsExtensions{.eps,.pdf,.png,.jpg}
\else
\DeclareGraphicsExtensions{.eps}
\fi

\usepackage{amsopn}

\usepackage{amsmath,amssymb}
\usepackage{caption}

\usepackage{cleveref}

\newsiamremark{remark}{Remark}
\newsiamremark{hypothesis}{Hypothesis}
\crefname{hypothesis}{Hypothesis}{Hypotheses}

\usepackage{enumitem}
\setlist[enumerate]{label=$\rm{(\roman*)}$,leftmargin=\parindent}

\headers{Fast Krasnosel'ski\u{\i}-Mann algorithm}{R. I. Bo\c{t} and D.-K. Nguyen}

\title{Fast Krasnosel'ski\u{\i}-Mann algorithm\\ with a convergence rate of the fixed point iteration of $\lowercase{o} \left(\frac{1}{\lowercase{k}} \right)$\thanks{Received by the editors June 21, 2022; 
		accepted for publication (in revised form) August 22, 2023; 
		\funding{The work of the first author was supported by the Austrian Science Fund (FWF), projects W 1260 and P 34922-N. The work of the second author was supported
			by the Austrian Science Fund (FWF), projects P 34922-N.}}}

\author{Radu Ioan Bo\c{t}\thanks{Faculty of Mathematics, University of Vienna, Oskar-Morgenstern-Platz 1, 1090 Vienna, Austria 
		(\email{radu.bot@univie.ac.at}).}
	\and Dang-Khoa Nguyen\thanks{Faculty of Mathematics, University of Vienna, Oskar-Morgenstern-Platz 1, 1090 Vienna, Austria 
		(\email{dang-khoa.nguyen@univie.ac.at}).}}

\begin{document}

\maketitle

\begin{abstract}
The \KM (KM) algorithm is the most fundamental iterative scheme designed to find a fixed point of an averaged operator in the framework of a real Hilbert space,  since it lies at the heart of various numerical algorithms for solving monotone inclusions and convex optimization problems. We enhance the \KM algorithm with Nesterov's momentum updates and show that the resulting numerical method  exhibits a convergence rate for the fixed point residual of $o \left(\frac{1}{k} \right)$ while preserving the weak convergence of the iterates to a fixed point of the operator.  Numerical experiments illustrate the superiority of the resulting so-called Fast KM algorithm over various fixed point iterative schemes, and also its oscillatory behavior, which is a specific of Nesterov's momentum optimization algorithms.
\end{abstract}



\section{Introduction}

Let $\sH$ be a real Hilbert space with inner product $\left\langle \cdot , \cdot \right\rangle$ and induced norm $\left\lVert \cdot \right\rVert$.
In this paper we are interested in formulating a fast numerical method for solving the fixed point problem
\begin{equation}
	\label{intro:pb:fix}
	\textrm{ Find } x \in \sH \textrm{ such that } x = T \left( x \right) ,
\end{equation}
where $T \colon \sH \to \sH$ is a $\theta$-averaged operator with $\theta \in \left( 0 , 1 \right]$. Recall that an operator $R \colon \sH \to \sH$ is \emph{nonexpansive} if it is $1$-Lipschitz continuous, that is
\begin{equation*}
	\left\lVert R \left( x \right) - R \left( y \right) \right\rVert \leq \left\lVert x - y \right\rVert \quad \forall x , y \in \sH.
\end{equation*}
Then $T$ is said to be \emph{averaged with constant $\theta$} or  \emph{$\theta$-averaged} if there exists a nonexpansive operator $R \colon \sH \to \sH$ such that 
$$T = \left( 1 - \theta \right) \Id + \theta R,$$
where $\Id \colon \sH \to \sH$ denotes the identity mapping on $\sH$. Obviously, an operator $T$ is nonexpansive if it is at least $1$-averaged.
We denote the set of all \emph{fixed points} of $T$ by $\sol := \left\lbrace x \in \sH  \ | \ x = T \left( x \right) \right\rbrace$.

The most naive approach when looking for a fixed point of $T$ is the following process, also called \emph{\BP iteration},
\begin{equation}
	\label{algo:BP}
	x_{k+1} := T \left( x_{k} \right) \quad \forall k \geq 0 ,
\end{equation}
where $x_0 \in \sH$ is a starting point.

According to the \BP fixed point theorem, if $T$ is a \emph{contraction}, namely, $T$ is Lipschitz continuous with modulus $\delta \in \left[ 0 , 1 \right)$, then the sequence $(x_k)_{k \geq 0}$ generated by \eqref{algo:BP} converges strongly to the unique fixed point of $T$ with linear convergence rate.  If $T$ is just nonexpansive, then this statement is no longer true. To illustrate this, it is enough to choose $T = - \Id$ and $x_{0} \neq 0$,  in which case the \BP iteration not only fails approach a fixed point of $T$, but also generates a sequence that does not satisfy the \emph{asymptotic regularity property}.  We say that 
the sequence $(x_k)_{k \geq 0}$ satisfies the asymptotic regularity property if the difference $x_{k} - T \left( x_{k} \right)$ converges strongly to $0$ as $k$ tends to $+ \infty$. This property is crucial for guaranteeing the convergence of the iterates, as we will see later.

In order to overcome the restrictive contraction assumption on $T$,  Krasnosel’ski\u{\i} proposed in \cite{Krasnoselskii} to apply the \BP iteration \eqref{algo:BP} to the operator $\frac{1}{2} \Id + \frac{1}{2} T$ instead of $T$. Following on the idea of using convex combinations the so-called \emph{\KM (KM) iteration}
\begin{equation}
	\label{algo:KM}
	x_{k+1} := \left( 1 - s_{k} \right) x_{k} + s_{k} T \left( x_{k} \right) \quad \forall k \geq 0,
\end{equation}
where $\left(s_{k} \right)_{k \geq 0}$ is a sequence in $\left( 0 , 1 \right]$,  emerged. It turned out that a fundamental step in proving the convergence of the iterates of \eqref{algo:KM} is to show that $x_{k} - T \left( x_{k} \right) \to 0$ as $k \to + \infty$,  as it was done by Browder and  Petryshyn in \cite{Browder-Petryshyn} in the constant case $s_{k} \equiv s \in \left( 0 , 1 \right)$.  The extension to nonconstant sequences was achieved by Groetsch in \cite{Groetsch} who proved that,  if $\sum_{k \geq 0} s_{k} \left( 1 - s_{k} \right) = + \infty$, then the asymptotic regularity property is satisfied. The weak convergence of the iterates was then studied in various settings in \cite{Groetsch,Ishikawa,Reich,Borwein-Reich-Shafrir,Bauschke-Combettes:book}.  Tikhonov regularization based techniques to improve the convergence of the iterates from weak to strong have been recently studied in \cite{Bot-Csetnek-Meier,Bot-Meier}.

By considering convex combinations with a fixed so-called \emph{anchor point} $x_{0} \in \sH$ one obtains the \emph{Halpern iteration} \cite{Halpern}
\begin{equation}
	\label{algo:Halpern}
	x_{k+1} := \left( 1 - s_{k} \right) x_{0} + s_{k} T \left( x_{k} \right) \quad \forall k \geq 0,
\end{equation}
a method that has recently attracted a lot of interest \cite{Lieder,Yoon-Ryu:21,Qi-Xu,Park-Ryu}. The asymptotic regularity property of this iterative scheme has been studied in \cite{Wittmann,Xu}.

Despite having ubiquitous applications in various fields,  the study of the computational complexity of fixed point iteration schemes is still limited. One natural measure to quantify this is by means of the rate of convergence of the \emph{fixed point residual} $\left\lVert x_{k} - T \left( x_{k} \right) \right\rVert$. Notice that the asymptotic regularity property does not automatically provide an explicit convergence rate. 

Sabach and  Shtern proved in \cite{Sabach-Shtern} for a  general form of the Halpern iteration that the rate of convergence of the fixed point residual is of $O \left( \frac{1}{k} \right)$. Lieder tightened this results in \cite{Lieder} by a constant factor, for the Halpern iteration with $s_{k} := 1 - \frac{1}{k+2}$ for every $k \geq 0$, whereas Park and Ryu proved in \cite{Park-Ryu} that the convergence rate of $O \left( \frac{1}{k} \right)$ is optimal for this iterative scheme, which means that it can not be improved in general. On the other hand, the convergence of the \KM iteration expressed in terms of the fixed point residual was in the nineties proved to be of $o \left(\frac{1}{\sqrt{k}} \right)$  in \cite{Baillon-Bruck} in the case of a constant sequence $\left(s_{k} \right)_{k \geq 0}$.  Later on, in the case of a nonconstant sequence, it was proved to be of $O \left(\frac{1}{\sqrt{k}} \right)$ in \cite{Cominetti-Soto-Vaisman,Liang-Fadili-Peyr}, and of $o \left(\frac{1}{\sqrt{k}}\right)$ in \cite{Bravo-Cominetti,Davis-Yin:2016,Matsushita}, whereas in \cite{Bot-Csetnek} it was shown that the asymptotic rate of convergence for the fixed point residual of the continuous time counterpart of the \KM iteration is of $o \left(\frac{1}{\sqrt{t}} \right)$. Recently,  Fierro, Maul\'en and Peypouquet proved in \cite{Fierro-Maulen-Peypouquet} that the rate of convergence of the fixed point residual of a general \emph{inertial \KM algorithm} is also of $o \left(\frac{1}{\sqrt{k}}\right)$. Noticeably,  Contreras and Cominetti showed in  \cite{Contreras-Cominetti} that in the \emph{Banach space setting} the lower bound of the  \KM iteration is $O \left(\frac{1}{\sqrt{k}} \right)$, which does not say anything about the lower bound in the Hilbert space setting.

For a family of general approaches aimed to ``accelerate'' the convergence of sequences relying on Shanks transformation and including Anderson acceleration, which can be applied also to fixed point problems, we refer to \cite{Brezinski-Redivo-Saad}.

In this paper we introduce an iterative method for solving the fixed point problem \eqref{intro:pb:fix} which exhibits a rate of convergence for the fixed point residual of $o \left( \frac{1}{k} \right)$ and guarantees the weak convergence of the iterates to a fixed point of $T$. The method is obtained  by enhancing the \KM iteration with Nesterov's momentum updates and follows via the temporal discretization of the second order dynamical system with vanishing damping term proposed in \cite{Bot-Csetnek-Nguyen} for solving monotone equations. The iterative scheme exploits the coercivity of the operator $\Id-T$ and has consequently a much more simple formulation than the Fast OGDA algorithm introduced in \cite{Bot-Csetnek-Nguyen} for solving monotone equations, which requires the construction of auxiliary sequences. Numerical experiments show that the resulting so-called Fast KM algorithm outperforms various fixed point iterative schemes including recently introduced ones using anchoring. The numerical experiments also illustrate the oscillatory behavior of the method,  which is a specific of algorithms with Nesterov's momentum updates.

As a by-product of our proposed approach we obtain several fast splitting methods for solving monotone inclusions. It is well-known that some of the most prominent splitting schemes result as particular instances of the \KM iteration, since they can be reduced to the solving of a fixed point problem governed by an average operator. This is the case for the \DR splitting \cite{Douglas-Rachford,Lions-Mercier}, the \FB splitting \cite{Lions-Mercier}, and the three operator splitting \cite{Davis-Yin,Davis}. For a comprehensive study of operator splitting schemes we refer to \cite{Bauschke-Combettes:book}. Recent contributions to the acceleration of the convergence of splitting methods have been made in  \cite{Kim,Yoon-Ryu:21,Lee-Kim,Tran-Dinh-Luo,Park-Ryu}.

\section{A fast \KM iteration}

In our approach, we rely on the simple observation that
\begin{equation*}
	x_{*} \in \sol \Leftrightarrow \left( \Id - T \right) \left( x_{*} \right) = 0,
\end{equation*}
which allows us to benefit from the recent development on a continuous fast method for solving monotone equations in \cite{Bot-Csetnek-Nguyen}. To be more specific, we have that $T$ is $\theta$-averaged if and only if $\Id - T$ is $\frac{1}{2 \theta}$-cocoercive \cite[Proposition 4.39]{Bauschke-Combettes:book}, that is
\begin{multline}
	\label{pre:coco}
	\left\langle x - y , \left( \Id - T \right) \left( x \right) - \left( \Id - T \right) \left( y \right) \right\rangle \\
	\geq \dfrac{1}{2 \theta} \left\lVert \left( \Id - T \right) \left( x \right) - \left( \Id - T \right) \left( y \right) \right\rVert ^{2} \geq 0 \quad \forall x , y \in \sH .
\end{multline}
From here one can immediately see that it follows immediately that $\Id - T$ is monotone. Furthermore, from the Cauchy-Schwarz inequality we can see that $\Id - T$ is at most $2 \theta$-Lipschitz continuous.

As a direct consequence of \eqref{pre:coco} we have that for every $x_{*} \in \Fix T$ it holds
\begin{equation*}
	\left\langle x - x_{*} , x - T \left( x \right) \right\rangle \geq \dfrac{1}{2 \theta} \left\lVert x - T \left( x \right) \right\rVert ^{2} \geq 0 \quad \forall x \in \sH .
\end{equation*}

The dynamical system studied in \cite{Bot-Csetnek-Nguyen}, formulated for the monotone equation $\left( \Id - T \right) \left( x \right) = 0$ and for constant time scaling term $\beta \left( t \right) \equiv 1$, reads
\begin{equation}	
	\label{pre:ds}
	\begin{cases}
		\ddot{x} \left( t \right) + \dfrac{\alpha}{t} \dot{x} \left( t \right) + \dfrac{d}{dt} \left( \Id - T \right) \left( x \left( t \right) \right) + \dfrac{\alpha}{2t} \left( \Id - T \right) \left( x \left( t \right) \right) = 0 \\
		x \left( t_{0} \right) = x_{0} \textrm{ and } \dot{x} \left( t_{0} \right) = \dot x_{0},
	\end{cases}
\end{equation}
where $\alpha \geq 2$, $\left( x_{0} , \dot x_{0} \right) \in \sH \times \sH$. The Lipschitz continuity of $\Id - T$ guarantees the existence and uniqueness of a strong global solution $x:[t_0,+\infty] \rightarrow \sH$, which means that $x$ and $\dot x$ are locally absolutely continuous,  $x \left( t_{0} \right) = x_{0}$, $\dot{x} \left( t_{0} \right) = \dot x_{0}$ and $x$ fulfills the first equation  in \eqref{pre:ds} almost everywhere.

We set the time scaling term equal to one since our aim is to derive via temporal discretization an explicit iterative fixed point scheme, whereas nonconstant time scaling terms are known to positively impact the convergence rates of implicit numerical algorithms; see \cite[Remark 2]{Bot-Csetnek-Nguyen} for a detailed discussion on this issue.

Consider the first-order formulation of the first equation in \eqref{pre:ds}
\begin{equation}
	\label{dis:ds-fo}
	\begin{cases}
		\dot{u} \left( t \right)	& = \left( 2 - \alpha \right) \left( \Id - T \right) \left( x \left( t \right) \right) \\
		u \left( t \right) 			& = 2 \left( \alpha - 1 \right) x \left( t \right) + 2t \dot{x} \left( t \right) + 2t\left( \Id - T \right) \left( x \left( t \right) \right).
	\end{cases}
\end{equation}
We fix a time step $s > 0$, set $s_{k} := s \left( k + 1 \right)$ for every $k \geq 1$, and approximate $x \left( s_{k} \right) \approx x_{k+1}$, and $u \left( s_{k} \right) \approx u_{k+1}$. The explicit finite-difference scheme for \eqref{dis:ds-fo} at time $t := s_{k}$ gives for every $k \geq 1$
\begin{equation}
	\label{dis:fd-fo}
	\begin{cases}
		\dfrac{u_{k+1} - u_{k}}{s} 	& = \left( 2 - \alpha \right) \left( \Id - T \right) \left( x_{k} \right) \\
		u_{k+1} 					& = 2 \left( \alpha - 1 \right) x_{k+1} + 2 \left( k + 1 \right) \left( x_{k+1} - x_{k} \right) + 2 s  \left( k + 1 \right) \left( \Id - T \right) \left( x_{k} \right),
	\end{cases}
\end{equation}
with the initialization  $u_{0} := x_{0} - s \dot x_{0}$ and $u_{1} := x_{0}$. Different to \cite{Bot-Csetnek-Nguyen}, where for the discretization of the argument of $\Id-T$ we used an auxiliary sequence, this time we can use $\left(x_{k} \right) _{k \geq 0}$. This is thanks to the stronger property of cocoercivity the operator $\Id-T$ is enhanced with and which will be reflected in the convergence analysis. We will see that this allows us not only to design a simpler algorithm, but also to consider larger step sizes than for the one proposed in \cite{Bot-Csetnek-Nguyen}.

Next we will simplify the sequence $\left(u_{k} \right)_{k \geq 0}$. The second equation in \eqref{dis:fd-fo} gives for every $k \geq 1$
\begin{equation*}
	u_{k} = 2 \left( \alpha - 1 \right) x_{k} + 2k \left( x_{k} - x_{k-1} \right) + 2sk \left( \Id - T \right) \left( x_{k-1} \right).
\end{equation*}
Taking the difference we obtain for every $k \geq 1$
\begin{align}
	u_{k+1} - u_{k}
	= & \ 2 \left( k + \alpha \right) \left( x_{k+1} - x_{k} \right) - 2k \left( x_{k} - x_{k-1} \right) + 2s \left( \Id - T \right) \left( x_{k} \right) \nonumber \\
	& + 2sk \left( \left( \Id - T \right) \left( x_{k} \right) - \left( \Id - T \right) \left( x_{k-1} \right) \right) \nonumber \\
	= &  \ \left( 2 - \alpha \right) s \left( \Id - T \right) \left( x_{k} \right) , \label{dis:d-u:pre}
\end{align}
where the last relation comes from the first equation in \eqref{dis:fd-fo}. After rearranging \eqref{dis:d-u:pre}, we obtain for every $k \geq 1$
\begin{align}	
	2 \left( k + \alpha \right) \left( x_{k+1} - x_{k} \right)
	& = 2k \left( x_{k} - x_{k-1} \right) - \alpha s \left( \Id - T \right) \left( x_{k} \right) \label{dis:d-u} \\
	& \quad - 2sk \left( \left( \Id - T \right) \left( x_{k} \right) - \left( \Id - T \right) \left( x_{k-1} \right) \right) . \nonumber
\end{align}
From here, we deduce that for every $k \geq 1$
\begin{align*}
	x_{k+1} 
	& = x_{k} + \dfrac{k}{k + \alpha} \left( x_{k} - x_{k-1} \right) - \dfrac{s \alpha}{2 \left( k + \alpha \right)} \left( \Id - T \right) \left( x_{k} \right) \nonumber \\
	& \quad - \dfrac{sk}{k + \alpha} \left( \left( \Id - T \right) \left( x_{k} \right) - \left( \Id - T \right) \left( x_{k-1} \right) \right) \nonumber \\
	& = \left( 1 - \dfrac{s \alpha}{2 \left( k + \alpha \right)} \right) x_{k} + \dfrac{\left( 1 - s \right) k}{k + \alpha} \left( x_{k} - x_{k-1} \right) \nonumber \\
	& \quad + \dfrac{s \alpha}{2 \left( k + \alpha \right)} T \left( x_{k} \right) + \dfrac{sk}{k + \alpha} \left( T \left( x_{k} \right) - T \left( x_{k-1} \right) \right) .
\end{align*}

Summing up, the algorithm we propose in this paper for solving \eqref{intro:pb:fix} has the following formulation.
\begin{algorithm}
	\caption{Fast KM algorithm}
	\label{algo:fKM}
	\begin{algorithmic}
		\STATE{Let $\alpha > 2, x_{0}, x_{1} \in \sH$ and $0 < s \leq \frac{1}{\theta}$.}
		\FOR{$k = 1, 2, \cdots$}
		\STATE{Compute}
		\begin{align}			
			x_{k+1}	& := \left( 1 - \dfrac{s \alpha}{2 \left( k + \alpha \right)} \right) x_{k} + \dfrac{\left( 1 - s \right) k}{k + \alpha} \left( x_{k} - x_{k-1} \right) \nonumber \\
			& \qquad + \dfrac{s \alpha}{2 \left( k + \alpha \right)} T \left( x_{k} \right) + \dfrac{sk}{k + \alpha} \left( T \left( x_{k} \right) - T \left( x_{k-1} \right) \right) . \label{algo:scheme}
		\end{align}
		\ENDFOR
	\end{algorithmic}
\end{algorithm}

\begin{remark}\label{remark1}
	For the step size choice $s:=1$ which is allowed for every $\theta$-averaged operator $T$ with  $\theta \in (0,1]$, our iterative scheme becomes
	\begin{equation}
		\label{algo:fNon}
		x_{k+1}	:= \left( 1 - \dfrac{\alpha}{2 \left( k + \alpha \right)} \right) x_{k} + \dfrac{\alpha}{2 \left( k + \alpha \right)} T \left( x_{k} \right) + \dfrac{k}{k + \alpha} \left( T \left( x_{k} \right) - T \left( x_{k-1} \right) \right) \quad \forall k \geq 1.
	\end{equation}
	Notice that for a \emph{nonexpansive} operator $T$, which corresponds to the case $\theta =1$, the value $s:=1$ is the largest step size that can be taken.
	
	The numerical algorithm \eqref{algo:fNon} can be interpreted as a \KM iteration enhanced with the extrapolation term $\frac{k}{k + \alpha} \left( T \left( x_{k} \right) - T \left( x_{k-1} \right) \right)$ which proves to have an accelerating effect on the convergence of the fixed point residual. We learn form here that, in order to improve the convergence rate while preserving the convergence of the iterates, one must address iterative schemes that go beyond the classical \emph{Mann iteration} \cite{Mann}.  The latter allows in the update rule only nonnegative coefficients for both the previous iterates and the operator evaluations at the previous iterates.
\end{remark}

\begin{remark}\label{remark2}
	A direct application of the explicit Fast OGDA method in \cite{Bot-Csetnek-Nguyen} to the solving of the monotone equation $\left( \Id - T \right) \left( x \right) = 0$ leads for given $\alpha > 2, x_{0}, x_{1}, y_{0} \in \sH$, $0 < s < \max \left\lbrace \frac{1}{4 \theta} , \frac{1 - \theta}{2 \theta} \right\rbrace$ to the following iterative scheme: for every $k \geq 1$ set
	\begin{subequations}
		\label{algo:fOGDA}
		\begin{align}
			y_{k} 	& := x_{k} + \left( 1 - \dfrac{\alpha}{k + \alpha} \right) \left( x_{k} - x_{k-1} \right) - \dfrac{\alpha s}{2 \left( k + \alpha \right)} \left( \Id - T \right) \left( y_{k-1} \right) \\
			x_{k+1}	& := y_{k} - \dfrac{s}{2} \left( 1 + \dfrac{k}{k + \alpha} \right) \left( \left( \Id - T \right) \left( y_{k} \right) - \left( \Id - T \right) \left( y_{k-1} \right) \right) .
		\end{align}
	\end{subequations}
	If $T$ is $\theta$-averaged, then $\Id - T$ is $L$-Lipschitz with $L := \min \left\lbrace 2 \theta , \frac{\theta}{1 - \theta} \right\rbrace$ (here we make the convention $\frac{1}{0} := + \infty$), thus the upper bound of the step size is $\frac{1}{2L} = \max \left\lbrace \frac{1}{4 \theta} , \frac{1 - \theta}{2 \theta} \right\rbrace$. Indeed, we already noticed that, since $\Id - T$ is $\frac{1}{2 \theta}$-cocoercive, it is at most $2 \theta$-Lipschitz continuous. On the other hand, $T$ is $\theta$-averaged if and only if (see \cite[Proposition 4.35]{Bauschke-Combettes:book})
	\begin{equation*}
		\left\lVert T \left( x \right) - T \left( y \right) \right\rVert ^{2} \leq \left\lVert x  - y \right\rVert ^{2} - \dfrac{1 - \theta}{\theta} \left\lVert \left( \Id - T \right) \left( x \right) - \left( \Id - T \right) \left( y \right) \right\rVert ^{2} \quad \forall x , y \in \sH ,
	\end{equation*}
	which implies that $\Id - T$ is Lipschitz continuous with modulus at most $\frac{\theta}{1 - \theta}$. 
	
	Noticeably,  the step size upper bound for algorithm \eqref{algo:fOGDA} is more restrictive compared to the one in the Fast KM iteration \eqref{algo:scheme}. In addition, \eqref{algo:scheme} has an obviously simpler formulation than \eqref{algo:fOGDA}. Even if one rewrites \eqref{algo:fOGDA} in terms of a single sequence $\left(y_{k} \right)_{k \geq 0}$, it would require $y_{k}, y_{k-1}$ and $y_{k-2}$ to compute $y_{k+1}$. In comparison, \eqref{algo:scheme} requires $x_{k}$ and $x_{k-1}$ to compute $x_{k+1}$.
\end{remark}

\section{Convergence analysis}

The fundamental tool of the convergence analysis is the following discrete energy function which, for fixed $x_{*} \in \sol$ and $0 \leq \lambda \leq \alpha-1$, is defined for every $k \geq 1$ as
\begin{align*}
	\E_{\lambda,k}
	:= & \ \dfrac{1}{2} \left\lVert 2 \lambda \left( x_{k} - x_{*} \right) + 2k \left( x_{k} - x_{k-1} \right) + \dfrac{1}{2 \left( \alpha - 1 \right)} \left( 3 \alpha - 2 \right) sk\left( \Id - T \right) \left( x_{k-1} \right) \right\rVert ^{2} \nonumber \\
	& + 2 \lambda \left( \alpha - 1 - \lambda \right) \left\lVert x_{k} - x_{*} \right\rVert ^{2} + \dfrac{1}{\alpha - 1} \left( \alpha - 2 \right) \lambda sk \left\langle x_{k} - x_{*} , \left( \Id - T \right) \left( x_{k-1} \right) \right\rangle \nonumber \\
	& + \dfrac{1}{8 \left( \alpha - 1 \right) ^{2}} \left( \alpha - 2 \right) \left( 3 \alpha - 2 \right) s^{2} k^{2} \left\lVert \left( \Id - T \right) \left( x_{k-1} \right) \right\rVert ^{2} .
\end{align*}
The discrete energy function is defined in analogy with the Lyapunov energy functions used in the study of continuous time dynamical systems associated with convex minimization problems and monotone equations (\cite{Attouch-Peypouquet-Redont, Bot-Csetnek-Nguyen}). While in convex minimization, the distance from the objective function at the current iterate to its minimal objective value plays the prominent role in the definition of the discrete energy function, in the present setting, this role is taken by $ \left\lVert \left( \Id - T \right) \left( x_{k-1} \right) \right\rVert^2$. The coefficient $k^2$ in front of this term suggests the rate at which we expect this term to converge, provided the energy sequence converges as $k \rightarrow +\infty$. The same reasoning applies to the third summand in the discrete energy function, while the first two summands will play an important role in proving the convergence of the iterates.

The properties of the discrete energy function are presented in the following lemma, with the proof deferred to the \cref{appendix:proof}.

\begin{lemma}
	\label{lem:dec}
	Let $x_{*} \in \sol$ and $\left(x_{k} \right)_{k \geq 0}$ be the sequence generated by \cref{algo:fKM}. 
	Then the following statements are true:
	\begin{enumerate}
		\item
		\label{lem:dec:dif}
		for $0 \leq \lambda \leq \alpha-1$ and every $k \geq 1$ it holds
		\begin{align}
			& \E_{\lambda,k+1} - \E_{\lambda,k} \label{dec:inq} \\
			\leq \ & 2 \left( 2 - \alpha \right) \lambda s \left\langle x_{k} - x_{*} , \left( \Id - T \right) \left( x_{k} \right) \right\rangle 
			+ \omega_{1} k \left\lVert x_{k+1} - x_{k} \right\rVert ^{2} \nonumber \\
			& \quad + s \left( \omega_{2} k + \omega_{3} \right) \left\langle x_{k+1} - x_{k} , \left( \Id - T \right) \left( x_{k} \right) \right\rangle 
			+ \omega_{4} s^{2} k \left\lVert \left( \Id - T \right) \left( x_{k} \right) \right\rVert ^{2} \nonumber \\
			& \quad + \dfrac{1}{\left( \alpha - 1 \right)} \left( \alpha - 2 \right) \left( s - \dfrac{1}{\theta} \right) sk^{2} \left\lVert \left( \Id - T \right) \left( x_{k} \right) - \left( \Id - T \right) \left( x_{k-1} \right) \right\rVert ^{2} , \nonumber 
		\end{align}
		where
		\begin{subequations}
			\label{dec:const}
			\begin{align}
				\omega_{1} 	& := 4 \left( \lambda + 1 - \alpha \right) \leq 0 , \\
				\omega_{2} 	& := \dfrac{1}{\alpha - 1} \Bigl( 4 \left( \alpha - 1 \right) \left( \lambda + 1 - \alpha \right) + \alpha \left( 2 - \alpha \right) \Bigr) \leq 0 , \\
				\omega_{3} 	& := \dfrac{1}{\alpha - 1} \left( 2 \alpha \left( \alpha - 1 \right) \left( \lambda + 1 - \alpha \right) + \alpha - 2 \left( \alpha - 1 \right) ^{2} + 2 \left( 2 - \alpha \right) \left( \alpha - 1 \right) \right) , \\
				\omega_{4} 	& := \dfrac{1}{2 \left( \alpha - 1 \right)} \left( 2 - \alpha \right) \left( 3 \alpha - 2 \right) \leq 0 ;
			\end{align}
		\end{subequations}
		
		\item 
		\label{lem:dec:bnd}
		for $0 \leq \lambda \leq \frac{3 \alpha}{4} - \frac{1}{2}$ the sequence $\left( \E_{\lambda,k} \right)_{k \geq 1}$ is nonnegative.
	\end{enumerate}
\end{lemma}

In the following lemma, we demonstrate that there exist infinitely many choices for the parameter $\lambda$ (depending on $\alpha$) for which an essential quantity in the expression on the right-hand side of \eqref{dec:inq} becomes non-positive after a finite number of iterations. As we will see in \cref{prop:sum}, this behaviour will lead to the convergence of the corresponding discrete energy function $ \E_{\lambda,k}$ as $k \rightarrow +\infty$. The proof of \cref{lem:trunc} can also be found in the \cref{appendix:proof}.

\begin{lemma}
	\label{lem:trunc}
	Let 
	\begin{subequations}
		\label{trunc:fea}
		\begin{align}
			\underline{\lambda} \left( \alpha \right) & := \dfrac{\alpha^{2}}{8 \left( \alpha - 1 \right)} + \dfrac{\alpha - 1}{2} - \dfrac{1}{8 \left( \alpha - 1 \right)} \left( \alpha - 2 \right) \sqrt{\left( \alpha - 2 \right) \left( 5 \alpha - 2 \right)} >0, \\
			\overline{\lambda} \left( \alpha \right) & := \min \left\lbrace \dfrac{3 \alpha}{4} - \dfrac{1}{2} , \dfrac{\alpha^{2}}{8 \left( \alpha - 1 \right)} + \dfrac{\alpha - 1}{2} + \dfrac{1}{8 \left( \alpha - 1 \right)} \left( \alpha - 2 \right) \sqrt{\left( \alpha - 2 \right) \left( 5 \alpha - 2 \right)} \right\rbrace .
		\end{align}
	\end{subequations}
	Then for every $\lambda$ satisfying $\underline{\lambda} \left( \alpha \right) < \lambda < \overline{\lambda} \left( \alpha \right)$ one can find an integer $k \left( \lambda \right) \geq 1$ with the property that the following inequality holds for every $k \geq k \left( \lambda \right)$
	\begin{align}
		R_{k} 
		:= \sqrt{\dfrac{5 \alpha - 2}{2 \left( 3 \alpha - 2 \right)}} \omega_{1} k \left\lVert x_{k+1} - x_{k} \right\rVert ^{2} + s \left( \omega_{2} k + \omega_{3} \right) \left\langle x_{k+1} - x_{k} , \left( \Id - T \right) \left( x_{k} \right) \right\rangle & \nonumber \\
		+ \sqrt{\dfrac{5 \alpha - 2}{2 \left( 3 \alpha - 2 \right)}} \omega_{4} s^{2} k \left\lVert \left( \Id - T \right) \left( x_{k} \right) \right\rVert ^{2} & \leq 0, \label{trunc:Rk}
	\end{align}
	where $\omega_1$, $\omega_2$, $\omega_3$ and $\omega_4$ are the constants defined in \eqref{dec:const}.
\end{lemma}

\begin{proposition}
	\label{prop:sum}
	Let $x_{*} \in \sol$ and $\left(x_{k} \right) _{k \geq 0}$ be the sequence generated by \cref{algo:fKM}.
	Then it holds
	\begin{subequations}
		\label{sum:state}
		\begin{align}
			\mysum_{k \geq 1} \left\langle x_{k} - x_{*} , \left( \Id - T \right) \left( x_{k} \right) \right\rangle & < + \infty , \label{sum:vi} \\
			\mysum_{k \geq 1} k \left\lVert x_{k+1} - x_{k} \right\rVert ^{2} & < + \infty , \label{sum:dx} \\
			\mysum_{k \geq 1} k \left\lVert \left( \Id - T \right) \left( x_{k} \right) \right\rVert ^{2} & < + \infty , \label{sum:I-T} \\
			\left( \dfrac{1}{\theta} - s \right) \mysum_{k \geq 1} k^{2} \left\lVert \left( \Id - T \right) \left( x_{k} \right) - \left( \Id - T \right) \left( x_{k-1} \right) \right\rVert ^{2} & < + \infty . \label{sum:d(I-T)}
		\end{align}
	\end{subequations}
	In addition, the sequence $\left( \E_{\lambda,k} \right) _{k \geq 1}$ converges for every $\underline{\lambda} \left( \alpha \right) < \lambda < \overline{\lambda} \left( \alpha \right)$, where the pair $\left( \underline{\lambda} \left( \alpha \right) , \overline{\lambda} \left( \alpha \right) \right)$ is defined in \eqref{trunc:fea}. Consequently, the sequence $\left(x_{k} \right) _{k \geq 0}$ is bounded.
\end{proposition}
\begin{proof}
	Let $\left( \underline{\lambda} \left( \alpha \right) , \overline{\lambda} \left( \alpha \right) \right)$ be the pair defined in \eqref{trunc:fea} and $\underline{\lambda} \left( \alpha \right) < \lambda < \overline{\lambda} \left( \alpha \right)$. By \cref{lem:trunc}, there exists an integer $k \left( \lambda \right) \geq 1$ such that for every $k \geq k(\lambda)$ it holds
	\begin{align*}
		& \ s \left( \omega_{2} k + \omega_{3} \right) \left\langle x_{k+1} - x_{k} , \left( \Id - T \right) \left( x_{k} \right) \right\rangle \\
		\leq & - \sqrt{\dfrac{5 \alpha - 2}{2 \left( 3 \alpha - 2 \right)}} \omega_{1} k \left\lVert x_{k+1} - x_{k} \right\rVert ^{2} 
		- \sqrt{\dfrac{5 \alpha - 2}{2 \left( 3 \alpha - 2 \right)}} \omega_{4} s^{2} k \left\lVert \left( \Id - T \right) \left( x_{k} \right) \right\rVert ^{2}.
	\end{align*}
	By plugging this inequality into \eqref{dec:inq} it follows that for every $k \geq k(\lambda)$ it holds
	\begin{align*}
		& \E_{\lambda,k+1} - \E_{\lambda,k} \nonumber \\
		\leq & \ 2 \left( 2 - \alpha \right) \lambda s \left\langle x_{k} - x_{*} , \left( \Id - T \right) \left( x_{k} \right) \right\rangle 
		+ \left( 1 - \sqrt{\dfrac{5 \alpha - 2}{2 \left( 3 \alpha - 2 \right)}} \right) \omega_{1} k \left\lVert x_{k+1} - x_{k} \right\rVert ^{2} \nonumber \\
		& \ + \left( 1 - \sqrt{\dfrac{5 \alpha - 2}{2 \left( 3 \alpha - 2 \right)}} \right) \omega_{4} s^{2} k \left\lVert \left( \Id - T \right) \left( x_{k} \right) \right\rVert ^{2} \nonumber \\
		& \ + \dfrac{1}{\left( \alpha - 1 \right)} \left( \alpha - 2 \right) \left( s - \dfrac{1}{\theta} \right) sk^{2} \left\lVert \left( \Id - T \right) \left( x_{k} \right) - \left( \Id - T \right) \left( x_{k-1} \right) \right\rVert ^{2} .
	\end{align*}
	Taking into account that $\omega_{1}, \omega_{4} \leq 0$ (see \eqref{dec:const}), and  $\frac{5 \alpha - 2}{2 \left( 3 \alpha - 2 \right)} < 1$, we can apply \cref{lem:quasi-Fej} to obtain the summability results in \eqref{sum:state} as well as the fact that the sequence $\left(\E_{\lambda,k} \right) _{k \geq 1}$ is convergent. Since $0 < \underline{\lambda} \left( \alpha \right) < \lambda < \overline{\lambda} \left( \alpha \right) < \alpha - 1$, the boundedness of $\left(x_{k} \right)_{k \geq 0}$ will then follow from the definition of the discrete energy function $\left(\E_{\lambda,k} \right)_{k \geq 1}$.
\end{proof}

Next we will show the  convergence of the sequence of iterates. The proof relies on the Opial Lemma (see \cref{lem:Opial:dis}) and the \emph{demiclosedness principle} for nonexpansive operators. According to this principle, if $(z_{k})_{k \geq 0} \subseteq \sH$ is a sequence which converges weakly to  $z \in \sH$ such that $z_{k} - T \left( z_{k} \right)$ converges strongly to $0$ as $k \rightarrow +\infty$, then $z \in \sol$ (see \cite[Corollary 4.28]{Bauschke-Combettes:book}).
\begin{theorem}
	\label{thm:conv}
	Let $\left(x_{k} \right)_{k \geq 0}$ be the sequence generated by \cref{algo:fKM}.
	Then $\left( x_{k} \right)_{k \geq 0}$ converges weakly to an element in $\sol$ as $k \rightarrow +\infty$.
\end{theorem}
\begin{proof}
	Let $x_{*} \in \sol$, $\left( \underline{\lambda} \left( \alpha \right) , \overline{\lambda} \left( \alpha \right) \right)$ be the pair defined in \eqref{trunc:fea} and $\underline{\lambda} \left( \alpha \right) < \lambda < \overline{\lambda} \left( \alpha \right)$. By the definition we have for every $k \geq 1$
	\begin{align}
		\label{defi:Ek:eq}
		\E_{\lambda,k} 
		= & \ 2 \lambda \left\langle x_{k} - x_{*} , 2k \left( x_{k} - x_{k-1} \right) + \dfrac{1}{2 \left( \alpha - 1 \right)} \left( 3 \alpha - 2 \right) sk \left( \Id - T \right) \left( x_{k-1} \right) \right\rangle \\
		& + 2 \lambda \left( \alpha - 1 \right) \left\lVert x_{k} - x_{*} \right\rVert ^{2} + \dfrac{1}{\alpha - 1} \left( \alpha - 2 \right) \lambda sk \left\langle x_{k} - x_{*} , \left( \Id - T \right) \left( x_{k-1} \right) \right\rangle \nonumber \\
		& + \dfrac{k^{2}}{2} \left\lVert 2 \left( x_{k} - x_{k-1} \right) + \dfrac{1}{2 \left( \alpha - 1 \right)} \left( 3 \alpha - 2 \right) s \left( \Id - T \right) \left( x_{k-1} \right) \right\rVert ^{2} \nonumber \\
		& + \dfrac{1}{8 \left( \alpha - 1 \right) ^{2}} \left( \alpha - 2 \right) \left( 3 \alpha - 2 \right) s^{2} k^{2} \left\lVert \left( \Id - T \right) \left( x_{k-1} \right) \right\rVert ^{2},\nonumber
	\end{align}
	which implies for every $\underline{\lambda} \left( \alpha \right) < \lambda_{1} < \lambda_{2} < \overline{\lambda} \left( \alpha \right)$ and every $k \geq 1$
	\begin{align}
		\label{conv:im:Ed-lambad}
		\E_{\lambda_{2}, k} - \E_{\lambda_{1}, k} 
		& = 4 \left( \lambda_{2} - \lambda_{1} \right) \Big( k \left\langle x_{k} - x_{*} , x_{k} - x_{k-1} + s \left( \Id - T \right) \left( x_{k-1} \right) \right\rangle \\
		& \qquad\qquad\qquad\qquad\qquad + \dfrac{1}{2} \left( \alpha - 1 \right) \left\lVert x_{k} - x_{*} \right\rVert ^{2} \Big) . \nonumber
	\end{align}
	For every $k \geq 1$ we set
	\begin{align}
		p_{k}	& := \dfrac{1}{2} \left( \alpha - 1 \right) \left\lVert x_{k} - x_{*} \right\rVert ^{2} + k \left\langle x_{k} - x_{*} , x_{k} - x_{k-1} + s \left( \Id - T \right) \left( x_{k-1} \right) \right\rangle , \label{defi:p-k} \\
		q_{k}	& := \dfrac{1}{2} \left\lVert x_{k} - x_{*} \right\rVert ^{2} + s \mysum_{i = 1}^{k} \left\langle x_{i} - x_{*} , \left( \Id - T \right) \left( x_{i-1} \right) \right\rangle. \label{defi:q-k}
	\end{align}
	We  notice that for every $k \geq 2$
	\begin{align*}
		q_{k} - q_{k-1}
		& = \left\langle x_{k} - x_{*} , x_{k} - x_{k-1} \right\rangle - \dfrac{1}{2} \left\lVert x_{k} - x_{k-1} \right\rVert ^{2} + s \left\langle x_{k} - x_{*} , \left( \Id - T \right) \left( x_{k} \right) \right\rangle ,
	\end{align*}
	and thus
	\begin{multline*}
		\left( \alpha - 1 \right) q_{k} + k \left( q_{k} - q_{k-1} \right) \\
		= p_{k} + \left( \alpha - 1 \right) s \mysum_{i = 1}^{k} \left\langle x_{i} - x_{*} , \left( \Id - T \right) \left( x_{i-1} \right) \right\rangle - \dfrac{k}{2} \left\lVert x_{k} - x_{k-1} \right\rVert ^{2} .
	\end{multline*}
	Since the discrete energy function converges for every $\underline{\lambda} \left( \alpha \right) < \lambda_{1} < \lambda_{2} < \overline{\lambda} \left( \alpha \right)$, we obtain that $\lim_{k \to + \infty} \left( \E_{\lambda_{2}, k} - \E_{\lambda_{1}, k}\right) \in \sR$ exists. This implies in view of \eqref{conv:im:Ed-lambad} and \eqref{defi:p-k} that
	\begin{equation}
		\label{conv:lim-p-k}
		\lim\limits_{k \to + \infty} p_{k} \in \sR \textrm{ exists}.
	\end{equation}	
	Moreover, thanks to the triangle inequality and the statements \eqref{sum:vi} - \eqref{sum:I-T} in \cref{prop:sum}, we have for every $k \geq 1$
	\begin{align*}
		& \mysum_{i = 1}^{k} \left\lvert \left\langle x_{i} - x_{*} , \left( \Id - T \right) \left( x_{i-1} \right) \right\rangle \right\rvert \nonumber \\
		\leq \ 	& \mysum_{i = 1}^{k} \left\lvert \left\langle x_{i} - x_{i-1} , \left( \Id - T \right) \left( x_{i-1} \right) \right\rangle \right\rvert + \mysum_{i = 1}^{k} \left\langle x_{i-1} - x_{*} , \left( \Id - T \right) \left( x_{i-1} \right) \right\rangle \nonumber \\
		\leq \ 	& \dfrac{1}{2} \mysum_{i = 1}^{k} \left\lVert x_{i} - x_{i-1} \right\rVert ^{2} + \dfrac{1}{2} \mysum_{i = 1}^{k} \left\lVert \left( \Id - T \right) \left( x_{i-1} \right) \right\rVert ^{2} + \mysum_{i = 1}^{k} \left\langle x_{i-1} - x_{*} , \left( \Id - T \right) \left( x_{i-1} \right) \right\rangle \nonumber \\
		\leq \ 	& \dfrac{1}{2} \mysum_{i \geq 1} \left\lVert x_{i} - x_{i-1} \right\rVert ^{2} + \dfrac{1}{2} \mysum_{i \geq 1} \left\lVert \left( \Id - T \right) \left( x_{i-1} \right) \right\rVert ^{2} + \mysum_{i \geq 1} \left\langle x_{i-1} - x_{*} , \left( \Id - T \right) \left( x_{i-1} \right) \right\rangle \nonumber \\
		< \ 	& + \infty .
	\end{align*}
	This means that the series $\sum_{i = 1}^{k} \left\langle x_{i} - x_{*} , \left( \Id - T \right) \left( x_{i-1} \right) \right\rangle$ is absolutely convergent.
	In addition, due to \eqref{sum:dx},
	\begin{equation*}
		\lim\limits_{k \to + \infty} k \left\lVert x_{k} - x_{k-1} \right\rVert ^{2} = 0,
	\end{equation*}
	which implies that
	\begin{equation*}
		\lim\limits_{k \to + \infty} \left (\left( \alpha - 1 \right) q_{k} + k \left( q_{k} - q_{k-1} \right) \right) \in \sR \textrm{ exists}.
	\end{equation*}
	From \cref{prop:sum}, we have that $\left( x_{k} \right)_{k \geq 0}$ is bounded, hence $\left(q_{k} \right)_{k \geq 1}$ is also bounded. This allows us to apply \cref{lem:lim-u-k} to conclude that $\lim_{k \to + \infty} q_{k} \in \sR$ also exists. Once again,  by the definition of $q_{k}$ in \eqref{defi:q-k} and the fact that the sequence $$\left( \sum_{i = 1}^{k} \left\langle x_{i-1} - x_{*} , \left( \Id - T \right) \left( x_{i-1} \right) \right\rangle \right)_{k \geq 1}$$ converges, it follows that $\lim_{k \to + \infty} \left\lVert x_{k} - x_{*} \right\rVert \in \sR$ exists. In other words, the hypothesis \labelcref{lem:Opial:dis:i} in Opial Lemma (see \cref{lem:Opial:dis}) is fulfilled.
	
	Now let $\overline{x}$ be a weak sequential cluster point of $\left( x_{k} \right) _{k \geq 0}$, meaning that there exists a subsequence $\left(x_{k_{n}} \right) _{n \geq 0}$ such that
	\begin{equation*}
		x_{k_{n}} \ \mbox{converges weakly to} \ \overline{x} \textrm{ as } n \to + \infty .
	\end{equation*}
	On the other hand, according to \eqref{sum:I-T},
	\begin{equation*}
		\left( \Id - T \right) \left( x_{k_{n}} \right) \ \mbox{converges strongly to} \ 0 \textrm{ as } n \to + \infty .
	\end{equation*}
	Due to the demiclosedness principle we conclude from here that $\overline{x} \in \sol$. This shows that the hypothesis \labelcref{lem:Opial:dis:ii} in Opial Lemma is also fulfilled, and completes the proof.
\end{proof}

The following results proves the convergence rate of the Fast KM algorithm in terms of the discrete velocity and fixed point residual.

\begin{theorem}
	\label{thm:rate}
	Let $\left( x_{k} \right) _{k \geq 0}$ be the sequence generated by \cref{algo:fKM}. Then it holds
	\begin{equation*}
		\left\lVert x_{k} - x_{k-1} \right\rVert = o \left( \dfrac{1}{k} \right) \qquad \textrm{ and } \qquad \left\lVert x_{k-1} - T \left( x_{k-1} \right) \right\rVert = o \left( \dfrac{1}{k} \right) \textrm{ as } k \to + \infty.
	\end{equation*}
\end{theorem}
\begin{proof}
	Let $x_{*} \in \sol$, $\left( \underline{\lambda} \left( \alpha \right) , \overline{\lambda} \left( \alpha \right) \right)$ be the pair defined in \eqref{trunc:fea} and $\underline{\lambda} \left( \alpha \right) < \lambda < \overline{\lambda} \left( \alpha \right)$. According to \cref{prop:sum}, the sequence $\left(\E_{\lambda,k} \right)_{k \geq 1}$ converges.
	
	From \eqref{defi:Ek:eq} and \eqref{defi:p-k} we have that for every $k \geq 1$
	\begin{align*}
		\E_{\lambda,k}
		& = 4 \lambda p_{k} + \dfrac{k^{2}}{2} \left\lVert 2 \left( x_{k} - x_{k-1} \right) + \dfrac{1}{2 \left( \alpha - 1 \right)} \left( 3 \alpha - 2 \right) s \left( \Id - T \right) \left( x_{k-1} \right) \right\rVert ^{2} \nonumber \\
		& \quad + \dfrac{1}{8 \left( \alpha - 1 \right) ^{2}} \left( \alpha - 2 \right) \left( 3 \alpha - 2 \right) s^{2} k^{2} \left\lVert \left( \Id - T \right) \left( x_{k-1} \right) \right\rVert ^{2} .
	\end{align*}
	We set for every $k \geq 1$
	\begin{align*}
		h_{k} 
		& := \dfrac{k^{2}}{2} \left( \left\lVert 2 \left( x_{k} - x_{k-1} \right) + \dfrac{1}{2 \left( \alpha - 1 \right)} \left( 3 \alpha - 2 \right) s \left( \Id - T \right) \left( x_{k-1} \right) \right\rVert ^{2} \right. \nonumber \\ 
		& \qquad\qquad\qquad \left. + \dfrac{1}{4 \left( \alpha - 1 \right) ^{2}} \left( \alpha - 2 \right) \left( 3 \alpha - 2 \right) s^{2} \left\lVert \left( \Id - T \right) \left( x_{k-1} \right) \right\rVert ^{2} \right) ,
	\end{align*}
	so that $\E_{\lambda,k} = 4 \lambda p_{k} + h_{k}$.
	Furthermore, since $\lim_{k \to + \infty} \E_{\lambda,k} \in \sR$ and $\lim_{k \to + \infty} p_{k} \in \sR$ (see also \eqref{conv:lim-p-k}), it holds
	\begin{equation*}
		\lim\limits_{k \to + \infty} h_{k} \in \sR \textrm{ exists}.
	\end{equation*}
	On the other hand, in view of \eqref{sum:dx} and \eqref{sum:I-T} in \cref{prop:sum} we have
	\begin{multline*}
		\mysum_{k \geq 1} \dfrac{1}{k} h_{k} \leq 4 \mysum_{k \geq 1} k \left\lVert x_{k} - x_{k-1} \right\rVert ^{2} \\
		+ \dfrac{1}{8 \left( \alpha - 1 \right) ^{2}} \left( 3 \alpha - 2 \right) \left( 7 \alpha - 6 \right) s^{2} \mysum_{k \geq 1} k \left\lVert \left( \Id - T \right) \left( x_{k-1} \right) \right\rVert ^{2} < + \infty .
	\end{multline*}
	Consequently, $\lim_{k \to + \infty} h_{k} = 0$, which yields
	\begin{align*}
		& \lim\limits_{k \to \infty} k \left\lVert 2 \left( x_{k} - x_{k-1} \right) + \dfrac{1}{2 \left( \alpha - 1 \right)} \left( 3 \alpha - 2 \right) s \left( \Id - T \right) \left( x_{k-1} \right) \right\rVert \\
		= & \lim\limits_{k \to \infty} k \left\lVert \left( \Id - T \right) \left( x_{k-1} \right) \right\rVert = 0 .
	\end{align*}
	This immediately implies $\lim_{k \to + \infty} k \left\lVert x_{k} - x_{k-1} \right\rVert = 0$.
\end{proof}

\begin{remark}\label{remark8}
	In \cite{Park-Ryu}, Park and Ryu established a fixed point residual lower bound of $O(\frac{1}{k})$ for various fixed point iterations designed to find a fixed point of a $1$-Lipschitz continuous operator. In the following, we will explain that this statement is not in contradiction with the convergence rate statement in \cref{thm:rate}. 
	
	According to \cite[Theorem 4.6]{Park-Ryu}, for given $K \geq 2$, $d \geq K$ and every initial point $x_0 \in \sR^d$, there exists an $1$-Lipschitz continuous operator $T : \sR^d \rightarrow \sR^d$ with $x_{*} \in \sol$ such that the inequality
	\begin{equation}\label{lowerbound}
		\|x_{K-1} - T(x_{K-1}) \| \geq \frac{2\|x_0-x_*\|}{K},
	\end{equation}
	holds for every iterates $(x_k)_{k = 0, ..., K-1}$ satisfying
	\begin{equation}\label{eqspan}
		x_k \in x_0 + \spa \{x_0-T(x_0), ..., x_{k-1} - T(x_{k-1})\}
	\end{equation}
	for $k = 1, ..., K-1$. It is evident that the sequence generated by the Fast KM algorithm with $x_1:=x_0$ fulfills \eqref{eqspan}. Consequently, there exists such an $1$-Lipschitz continuous operator $T : \sR^d \rightarrow \sR^d$ that fulfills \eqref{lowerbound} for the sequence of iterates generated by the Fast KM algorithm with $x_1:=x_0$. In contrast, according to \cref{thm:rate}, for the same operator (as it is the case for every  $1$-Lipschitz continuous operator), it holds $k  \|x_{k-1} - T(x_{k-1}) \| \rightarrow 0$ as $k \rightarrow +\infty$, which means that there exists $k_0 > K$ such that 
	\begin{equation*}
		\|x_{k-1} - T(x_{k-1}) \| < \frac{2\|x_0-x_*\|}{k} 
	\end{equation*}
	for every $k \geq k_0$.
\end{remark}

\begin{remark}\label{remark9}
	In \cite{Contreras-Cominetti}, Contreras and Cominetti presented an example of a $1$-Lipschitz continuous operator defined on a Banach space with the property that the fixed point residual of the \textit{Mann iteration}  is bounded from below by $O(\frac{1}{k})$. For given initial points $y_0$ and $x_0$, the Mann iterates are defined for every $k \geq 0$ recursively as follows
	\begin{equation*}
		x_{k+1}:= s^{k+1}_0 y_0 + \sum_{i=0}^{k} s^{k+1}_{i+1} T(x_i),
	\end{equation*}
	where $s_{i}^{k+1} \geq 0$ for every $i=0, ..., k+1$ and $\sum_{i=0}^{k+1} s_i^{k+1} =1$. 
	
	As observed in \cref{remark1}, the Fast KM algorithm is not a Mann iteration since not all the weights are nonnegative, although  they do sum up to $1$. This observation suggests that in order to obtain a fixed point residual rate of $o(\frac{1}{k})$ in Banach spaces, one might have to go beyond general Mann iterations and allow, for instance, negative weights in the iterative scheme.
\end{remark}

\section{Application to several splitting algorithms}
\label{sec:split}

In the light of the fact that fixed point methods lie at the heart of important splitting iterative schemes for monotone inclusions, we will discuss in this section how the Fast KM algorithm impacts the latter. In addition, we will review some recent acceleration approaches of splitting algorithms from the literature.

Consider the following monotone inclusion problem
\begin{equation}
	\label{intro:pb:mi}
	\mbox{Find} \ x \in \sH \ \mbox{such that} \ 0 \in A \left( x \right) + B \left( x \right) + C \left( x \right) ,
\end{equation}
where $A, B \colon \sH \to 2^{\sH}$ are  set-valued maximally monotone operators and $C : \sH \to \sH$ is a $\beta$-cocoercive operator with $\beta > 0$. 

Davis and Yin introduced in \cite{Davis-Yin} the following operator
\begin{equation}
	\label{os:T-DY}
	T_{DY} : \sH \to \sH, \quad T_{DY} := J_{\gamma A} \circ \left( 2 J_{\gamma B} - \Id - \gamma C \circ J_{\gamma B} \right) + \Id - J_{\gamma B} ,
\end{equation}
where $0 < \gamma \leq 2\beta$ and $J_{\gamma A} := \left( \Id + \gamma A \right) ^{-1}$ denotes the \emph{resolvent operator} of $\gamma A$ with constant $\gamma$. The set of zeros of $A+B+C$, denoted by $\zer \left( A+B+C \right)$, can be characterized in terms of  $T_{DY}$ by (see \cite[Lemma 2.2]{Davis-Yin})
\begin{equation*}
	\zer \left( A+B+C \right) = J_{\gamma B} \left( \Fix T_{DY} \right) .
\end{equation*}

The \KM iteration applied to $T_{DY}$ gives rise to the \emph{three-operator splitting method}
\begin{equation*}
	x_{k+1} := \left( 1 - s_{k} \right) x_{k} + s_{k} T_{DY} \left( x_{k} \right) \quad \forall k \geq 0,
\end{equation*}
where $x_0 \in \sH$ and $\left(s_{k} \right)_{k \geq 0} \subseteq \left( 0 , 2 - \frac{\gamma}{2 \beta} \right]$. The operator $T_{DY}$ is $\frac{2 \beta}{4 \beta - \gamma}$-averaged, here we use the convention $\beta := + \infty$ whenever $C \equiv 0$, in which case the operator is $\frac{1}{2}$-averaged. According to \cite{Huang-Ryu-Yin}, this constant is tight. It has been shown in \cite{Davis} that the convergence rate of three-operator splitting method is, as expected, of $O \left(\frac{1}{\sqrt{k}} \right)$.

The theoretical statements of the previous section applied to this particular setting lead to the following result.
\begin{corollary}
	\label{coro:fDY}
	Let $\alpha > 2, x_{0}, x_{1} \in \sH$, $0 < \gamma \leq 2\beta$ and $0 < s \leq 2 - \frac{\gamma}{2 \beta}$.
	For every $k \geq 1$ we set
	\begin{align*}
		x_{k+1}	:= & \left( 1 - \dfrac{s \alpha}{2 \left( k + \alpha \right)} \right) x_{k} + \dfrac{\left( 1 - s \right) k}{k + \alpha} \left( x_{k} - x_{k-1} \right) \\
		& + \dfrac{s \alpha}{2 \left( k + \alpha \right)} T_{DY} \left( x_{k} \right) + \dfrac{sk}{k + \alpha} \left( T_{DY} \left( x_{k} \right) - T_{DY} \left( x_{k-1} \right) \right). 
	\end{align*}
	Then the following statements are true:
	\begin{enumerate}
		\item 
		\label{coro:fDY:i}
		$\left( x_{k} \right)_{k \geq 0}$ converges weakly to an element $x_{*}$ in $\Fix T_{DY}$ such that $J_{\gamma B} \left( x_{*} \right)$ is a solution of \eqref{intro:pb:mi};
		
		\item 
		\label{coro:fDY:ii}
		it holds
		\begin{equation*}
			\left\lVert x_{k} - x_{k-1} \right\rVert = o \left( \dfrac{1}{k} \right) \quad \textrm{ and } \quad \left\lVert x_{k-1} - T_{DY} \left( x_{k-1} \right) \right\rVert = o \left( \dfrac{1}{k} \right) \textrm{ as } k \to + \infty.
		\end{equation*}
	\end{enumerate}	
\end{corollary}

In the following we will revisit some of the particular formulations of \eqref{intro:pb:mi} and of the corresponding underlying operator $T_{DY}$ also in order to emphasize the broad applicability of \cref{coro:fDY}.

\subsection*{Resolvent operator}
For $B \equiv C \equiv 0$, the problem \eqref{intro:pb:mi} reduces to 
\begin{equation*}
	\mbox{Find} \ x \in \sH \ \mbox{such that} \ 0 \in A \left( x \right),
\end{equation*}
and, for $\gamma >0$, 
$$T_{DY} = J_{\gamma A},$$ 
which is $\frac{1}{2}$-averaged.

The fixed point residual of the classical proximal point algorithm
\begin{equation*}
	x_{k+1} = J_{\gamma A} \left( x_{k} \right) \quad \forall k \geq 0,
\end{equation*}
is known to be in general of $O \left(\frac{1}{\sqrt{k}} \right)$, whereas Gu and Yang have shown in  \cite{Gu-Yang} that, for $\sH := \sR^{n}$,  it can be tightened to
\begin{equation*}
	\left\lVert J_{\gamma A} \left( x_{k} \right) - x_{k} \right\rVert = \begin{cases}
		O\left (\dfrac{1}{k} \right),	& \textrm{ if } n = 1 , \\
		O \left( \dfrac{1}{\sqrt{\left( 1 + \frac{1}{k} \right) ^{k} k}} \right), & \textrm{ if } n \geq 2 .
	\end{cases}
\end{equation*}

In the same setting of finite-dimensional Hilbert spaces, Kim proposed in \cite{Kim} (see also  \cite{Park-Ryu}) the following \emph{accelerated proximal point method}, which, given $y_{1} = x_{0} = x_{1} \in \sH$ and $\gamma >0$, reads: for every $k \geq 1$ set
\begin{align}
	\begin{split}
		\label{algo:APPM}
		y_{k+1} 	& := J_{\gamma A} \left( x_{k} \right) , \\
		x_{k+1}	& := y_{k+1} + \dfrac{k}{k+2} \left( y_{k+1} - y_{k} \right) - \dfrac{k}{k+2} \left( y_{k} - x_{k-1} \right) .
	\end{split}
\end{align}
Using the performance estimation approach, the the author proved that the method exhibits a convergence rate of the fixed point residual of $O\left (\frac{1}{k} \right)$. 

By comparison, the algorithm in \cref{coro:fDY} for $\gamma >0$, $0 < s \leq 2$ and $T_{DY} = J_{\gamma A}$ exhibits a convergence rate of the fixed point residual of $o\left (\frac{1}{k} \right)$. 

\subsection*{Forward-backward operator}
For $B \equiv 0$, the problem \eqref{intro:pb:mi} reduces to 
\begin{equation*}
	\mbox{Find} \ x \in \sH \ \mbox{such that} \ 0 \in A \left( x \right) + C(x),
\end{equation*}
and, for $0 < \gamma \leq 2 \beta$, 
$$T_{DY} := T_{FB} = J_{\gamma A} \circ (\Id - \gamma C),$$
which is $\frac{2 \beta}{4 \beta - \gamma}$-averaged.

The \KM iteration gives rise in this case to the classical \emph{forward-backward} algorithm, which is known to exhibit a convergence rate of the fixed point residual of $O\left (\frac{1}{\sqrt{k}} \right)$. By comparison, the algorithm in \cref{coro:fDY} for  $T_{DY} = T_{FB}$ exhibits a convergence rate of the fixed point residual of $o\left (\frac{1}{k} \right)$. 

\subsection*{\DR operator}
For $C \equiv 0$, the problem \eqref{intro:pb:mi} reduces to 
\begin{equation*}
	\mbox{Find} \ x \in \sH \ \mbox{such that} \ 0 \in A \left( x \right) + B(x),
\end{equation*}
and, for $\gamma >0$, 
$$T_{DY} := T_{DR} = J_{\gamma A} \circ (2 J_{\gamma B} - \Id)+ \Id - J_{\gamma B},$$ 
which is $\frac{1}{2}$-averaged.

The \KM iteration gives rise in this case to the classical \emph{\DR} algorithm (see \cite{Douglas-Rachford}, \cite{Lions-Mercier}), which is known to exhibit a convergence rate of the fixed point residual of $O\left (\frac{1}{\sqrt{k}} \right)$ (see \cite{He-Yuan}). By comparison, the algorithm in \cref{coro:fDY} for  $\gamma >0$, $0 < s \leq 2$ and $T_{DY} = T_{DR}$ exhibits a convergence rate of the fixed point residual of $o\left (\frac{1}{k} \right)$.

\subsection*{Recent contributions to acceleration approaches}
The idea of the Halpern iteration of considering in the iterative schemes convex combinations with an anchor point has been recently extensively exploited as it led to convergence rate improvements of numerical methods. This has been first done for fixed point iterations (\cite{Sabach-Shtern}, \cite{Lieder}), then for algorithms for solving monotone equations and minimax problems (\cite{Yoon-Ryu:21}, \cite{Lee-Kim}), and later on for variants of splitting algorithms like the forward-backward, the \DR and the three-operator splitting method (\cite{Qi-Xu,Tran-Dinh-Luo,Tran-Dinh,Yoon-Ryu}). 

Our method, however, relies on the continuous time approach from \cite{Bot-Csetnek-Nguyen} and uses the idea of Nesterov's momentum updates (\cite{Nesterov:83}).

\section{Numerical experiments}

\subsection{Proximal point type methods}\label{subsec51}

In order to illustrate the numerical performances of the Fast KM algorithm by comparison to other iterative schemes we consider first, for $n \geq 1$, the fixed point problem
$$\mbox{Find} \ x \in \sR^{2n} \ \mbox{such that} \ J_A(x) = x,$$
where $J_A : \sR^{2n} \rightarrow \sR^{2n}$ is the resolvent of the maximally monotone operator given by the matrix
\begin{equation*}
	A =  \frac{1}{M-1} \begin{pmatrix}
		\mO 	& \mI \\ - \mI 	& \mO
	\end{pmatrix} \in \sR^{2n \times 2n},
\end{equation*}
where $M$ is a positive constant, $\mI$ and $\mO$ denote the identity and the all-zeros matrix in $\sR^{n \times n}$, respectively. This operator has been used in the literature to illustrate the worse-case performance of the proximal point method and of the \BP iteration  (see \cite{Gu-Yang,Park-Ryu}). Notice that $x=0$ is the unique fixed point of $J_A$ and that $J_A$ is $\frac{1}{2}$-averaged,  therefore we take as step size $s:=2$.

\begin{figure}[!htb]
	\begin{center}
		\includegraphics[width=0.7\linewidth]{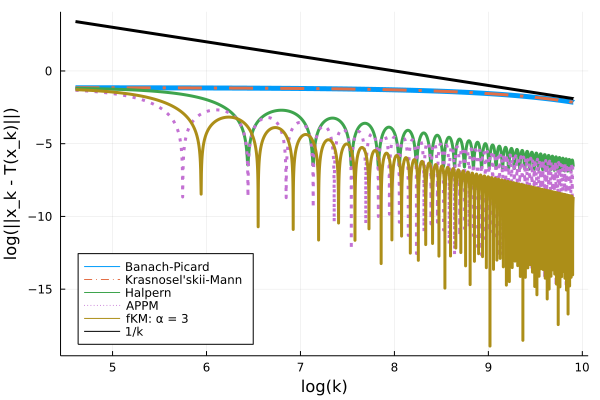}
		\caption{Comparison of the behaviour of the fixed point residual for different methods}
		\label{fig:fpall}
	\end{center}
\end{figure}

We solve the fixed point problem with the \BP iteration, which corresponds to the proximal point algorithm, the \KM iteration \eqref{algo:KM}, which corresponds to the relaxed proximal point algorithm, the Halpern iteration \eqref{algo:Halpern}, the accelerated proximal point method (APPM)  \eqref{algo:APPM}, and the Fast KM algorithm \eqref{algo:scheme} with $\alpha=3$.

\begin{figure}[!htb]
	\begin{center}
		\includegraphics[width=0.7\linewidth]{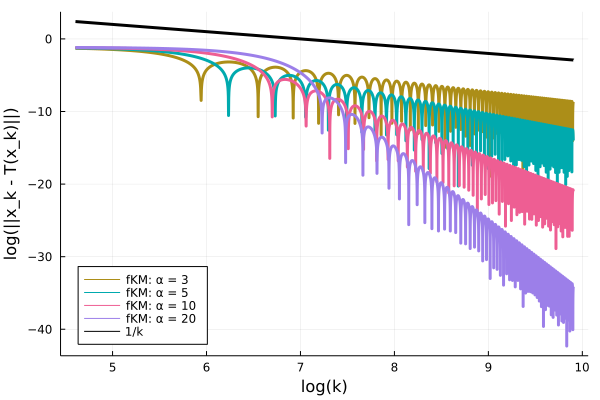}
		\caption{The parameter $\alpha$ influences the convergence behaviour of the Fast KM algorithm}
		\label{fig:fpf}
	\end{center}
\end{figure}

For all iterative methods we consider as starting $\begin{pmatrix}
	\vj_{n} \\ \vo_{n} 
\end{pmatrix} \in \sR^{2n}$, where $\vj_{n}$ and $\vo_{n}$ denote then all ones and all zeros vector in $\sR^{n}$, respectively.

In a first experiment, we run all these methods in case $n := 5000$. The values of the corresponding fixed point residuals are plotted in \cref{fig:fpall} in logarithmic scale. It is obvious that the Fast KM algorithm outperforms all other numerical algorithms.

In a second experiment, we run the Fast KM algorithm for different values of $\alpha$ in $\left\lbrace 3, 5, 10, 20 \right\rbrace$. The values of the corresponding fixed point residuals are plotted in \cref{fig:fpf} in logarithmic scale. As there is almost no difference between the methods in the first $100$ iterations, one can notice that the speed of convergence of the fixed point residual increases with increasing $\alpha$ and it is consistently faster than $o \left( 1/k \right)$. This phenomenon seems to be common for algorithms enhanced with Nesterov's momentum update (see also \cite{Bot-Csetnek-Nguyen}). However, whereas in the case of Nesterov’s acceleration algorithms for minimizing smooth and convex function the values of $\alpha$ are correlated with the speed of convergence of the objective function values, here this applies to the fixed point residual.

In addition, the plots of the fixed point residual exhibit a strong oscillatory behaviour, very similar to the behaviour of the objective function values of Nesterov’s acceleration algorithms for convex minimization. This is another evidence that Nesterov’s momentum improves the convergence behaviour of numerical algorithms beyond the  optimization setting (see also \cite{Bot-Csetnek-Nguyen}).

\begin{figure}[!htb]
	\minipage{0.48\textwidth}
	\includegraphics[width=\linewidth]{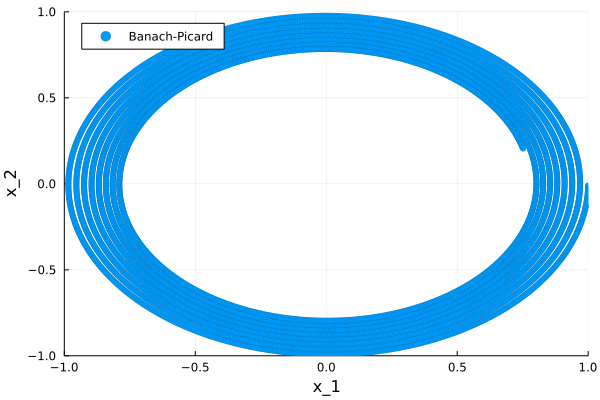}
	\caption{The trajectory of the \BP iteration/proximal point algorithm}
	\label{fig:traBP}
	\endminipage\hfill
	\minipage{0.48\textwidth}
	\includegraphics[width=\linewidth]{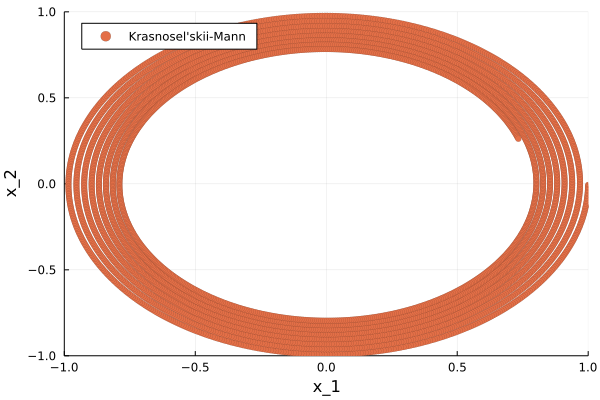}
	\caption{The trajectory of the \KM iteration/relaxed proximal point algorithm}
	\label{fig:traKM}
	\endminipage
\end{figure}

\begin{figure}[!htb]
	\minipage{0.48\textwidth}
	\includegraphics[width=\linewidth]{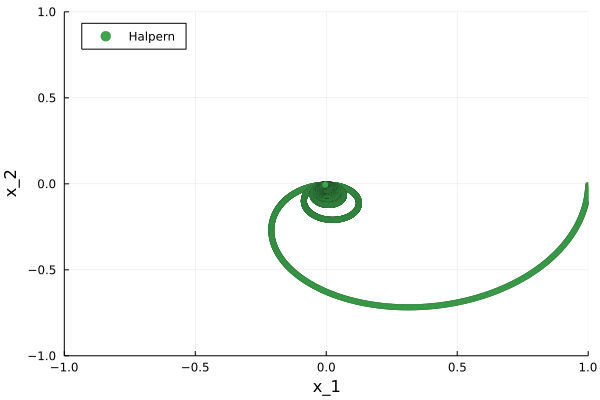}
	\caption{The trajectory of Halpern iteration}
	\label{fig:traHalpern}
	\endminipage\hfill
	\minipage{0.48\textwidth}
	\includegraphics[width=\linewidth]{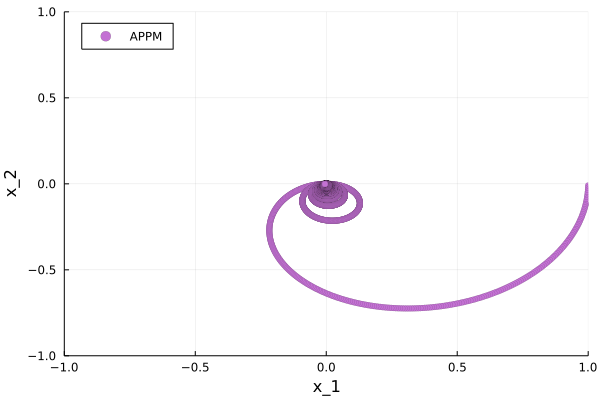}
	\caption{The trajectory of the accelerated proximal point method}
	\label{fig:traAPPM}
	\endminipage
\end{figure}

In \cref{fig:traBP} - \cref{fig:trafKM20} we plot the trajectories generated by all methods considered in the numerical experiments in case $n:=1$. From the convergence analysis we know that they all converge to the unique fixed point of the operator, however, as the plots show, after spiralling around it. It is also obvious that the spiralling effect in case of the Fast KM algorithm is less pronounced than for the other algorithms.

\begin{figure}[!htb]
	\minipage{0.48\textwidth}
	\includegraphics[width=\linewidth]{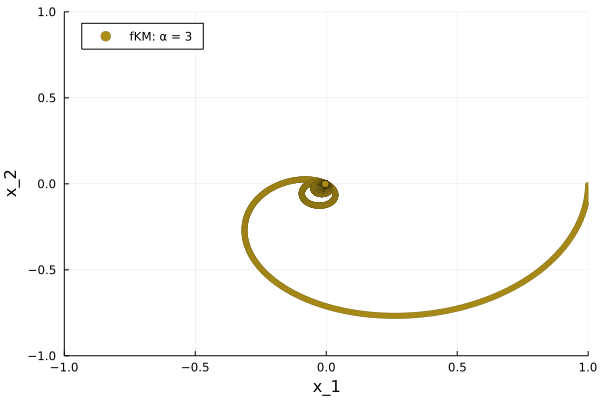}
	\caption{The trajectory of the Fast KM algorithm for $\alpha = 3$}
	\label{fig:trafKM3}
	\endminipage\hfill
	\minipage{0.48\textwidth}
	\includegraphics[width=\linewidth]{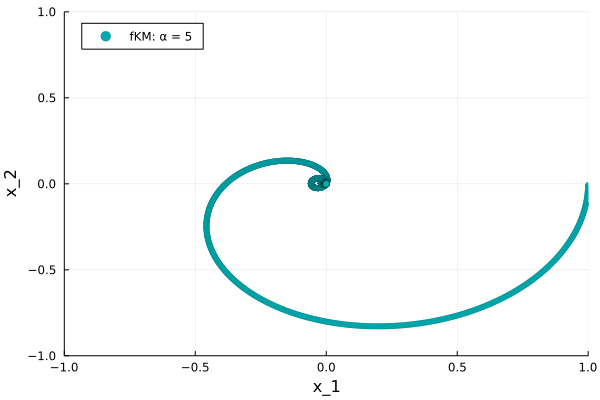}
	\caption{The trajectory of the Fast KM algorithm for $\alpha = 5$}
	\label{fig:trafKM5}
	\endminipage
\end{figure}

\begin{figure}[!htb]
	\minipage{0.48\textwidth}
	\includegraphics[width=\linewidth]{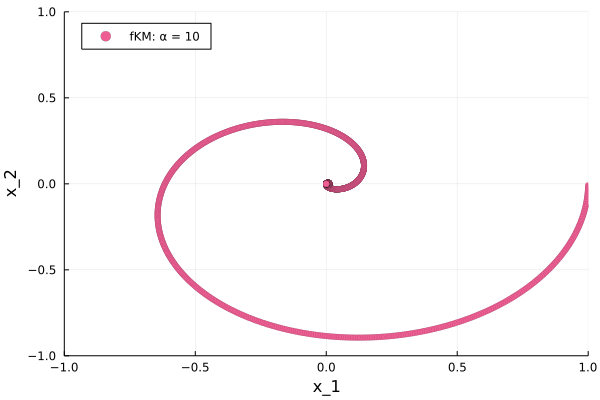}
	\caption{The trajectory of the Fast KM algorithm for $\alpha = 10$}
	\label{fig:trafKM10}
	\endminipage\hfill
	\minipage{0.48\textwidth}
	\includegraphics[width=\linewidth]{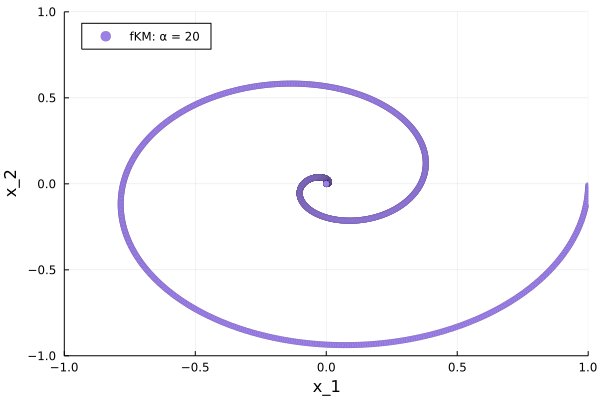}
	\caption{The trajectory of the Fast KM algorithm for $\alpha = 20$}
	\label{fig:trafKM20}
	\endminipage
\end{figure}

\subsection{\DR type methods}\label{subsec52}

For the second family of numerical experiments we consider the following feasibility problem
\begin{equation*}
	\mbox{Find } x \in \sR^{2n} \mbox{ such that } x \in \sR_{+}^{2n} \cap H_{u, \nu} ,
\end{equation*}
where $\sR_{+}^{2n}$ is the set of vectors in $\sR^{2n}$ with nonnegative entries and, for a given vector $u \in \sR^{2n}$ and real number $\nu \in \sR$,
\begin{equation*}
	H_{u, \nu} := \left\lbrace x \in \sR^{2n} \colon \left\langle x , u \right\rangle = \nu \right\rbrace .
\end{equation*}
If $\sR_{+}^{2n} \cap H_{u, \nu} \neq \emptyset$, then (see, for instance, \cite[Corollary 27.6]{Bauschke-Combettes:book})
\begin{equation*}
	\sR_{+}^{2n} \cap H_{u, \nu} = \zer \left( N_{\sR_{+}^{2n}} + N_{H_{u, \nu}} \right) \neq \emptyset ,
\end{equation*}
where $N_{D} : \sR^{2n} \rightrightarrows \sR^{2n}$ denotes the normal cone operator of a nonempty closed convex set $D \subseteq \sR^n$. The normal cone operator $N_{D}$ is a maximally monotone and its resolvent $J_{N_{D}}$ is nothing else than the projection $\proj_{D}$ onto the set $D$. 

The \DR (DR) algorithm is known as one of the most successful numerical method for solving such feasibility problems. In this concrete case it reads
$$x_{k+1}:= (1-s_k)x_k + s_kT_{DR}(x_k) \quad \forall k \geq 0,$$
where $x_0 \in {\sR^{2n}}, (s_k)_{k \geq 0} \subseteq (0,2]$ and $T_{DR}:=\proj_{\sR_{+}^{2n}} \circ(2\proj_{H_{u, \nu}} - \Id) + \Id - \proj_{H_{u, \nu}}$.

In the following we compare the performances of the DR algorithm for various choices for $\left( s_{k} \right) _{k \geq 0}  \subseteq (0,2]$ with the ones of the Halpern algorithm (\cite{Halpern}) and of the Fast KM algorithm, which make also use of the  \DR operator $T_{DR}$. For the Fast KM algorithm we consider $\alpha \in \left\lbrace 5, 10, 30, 100, 500 \right\rbrace$ and as a step size $s := 2$.

In the numerical experiments we generate for different values for the dimension $n \geq 1$ a number of $\mathtt{N_{test}}$ pairs $(u, \nu) \in \sR_{+}^{2n} \times \sR_+$, such that the intersection of $\sR_{+}^{2n}$ and $H_{u, \nu}$ is nonempty, and a number of $\mathtt{N_{init}}$ normally distributed starting points $x_{0} \in \sR^{2n}$, which we scale then by $100$. For each generated hyperplane $H_{u, \nu}$ and starting point $x_{0}$ we run several variants of the DR algorithm, the Fast KM method and the Halpern algorithm. The algorithms terminate either after $\mathtt{k_{\max}}$ iterations or once the following condition is fulfilled
\begin{equation}
	\label{intersection}
	\left\lVert \proj_{H_{u, \nu}} \left( x_{k} \right) - \proj_{\sR_{+}^{2n}} \left( \proj_{H_{u, \nu}} \left( x_{k} \right) \right) \right\rVert \leq \mathtt{Tol},
\end{equation}
where $\mathtt{Tol}$ denotes the tolerance error. This condition is motivated by the fact that for the \DR methods the so-called shadow sequence is the one that converges to a solution; in other words, \eqref{intersection} guarantees that $\proj_{H_{u, \nu}} \left( x_{k} \right)$ is close to the intersection $\sR_{+}^{2n} \cap H_{u, \nu}$. A trial fulfilling \eqref{intersection} before $\mathtt{k_{\max}}$ iterations will be counted as a successful attempt. In \cref{tab:small} and \cref{tab:large} we report the ratio of successfully solved problems and the average number of iterations the algorithms need until termination.

\cref{tab:small} shows the results for three settings determined by three choices for the triple $(n, \mathtt{N_{test}}, \mathtt{N_{init}})$ with $\mathtt{Tol} := 10^{-16}$ for the first two and $\mathtt{Tol} := 10^{-12}$ for the third one, 
and $\mathtt{k_{\max}} := 100$. For each setting we write in boldface the best values for the ratios and the average number of iterates for the DR algorithms. It is evident that the Fast KM algorithm outperforms in both criteria the best performing variants of the \DR algorithm and the Halpern algorithm already for $\alpha =30$,  and even more so for larger values of $\alpha$.

\begin{table}[!htb]
	\begin{center}
		\renewcommand{\arraystretch}{2}	
\tiny
		\begin{tabular}{l|l|c|l|c|l|c}
			\toprule
			{\bf $(n, \mathtt{N_{test}}, \mathtt{N_{init}})$} & \multicolumn{2}{|c|}{$\left( 1,10^2,10^4 \right)$} & \multicolumn{2}{|c|}{$\left( 5,10^2,10^4 \right)$} & \multicolumn{2}{|c}{$\left( 50,10^2,10^3 \right)$} \\
			\hline
			\multicolumn{1}{c|}{\textbf{method}} & \textbf{ratio} & \textbf{iterations} & \textbf{ratio} & \textbf{iterations} & \textbf{ratio} & \textbf{iterations} \\
			\midrule
			\texttt{DR} $\colon s_{k} = 1 - \frac{1}{k+2}$   				
			& $0.9939$	& $7.6181 \pm 6.26$ 	& $0.9887$	& $14.5431 \pm 12.54$ 	& $0.8456$	& $42.823 \pm 20.99$ \\
			\texttt{DR} $\colon s_{k} \equiv 1$           				
			& $0.9940$	& $\bf 4.8877 \pm 6.00$ 	& $0.9907$	& $\bf 11.6671 \pm 12.17$ 	& $0.8700$	& $39.1163 \pm 21.08$ \\
			\texttt{DR} $\colon s_{k} = 1 + \frac{1}{k+2}$   				
			& $0.9944$	& $5.9398 \pm 6.23$ 	& $0.9923$	& $12.0513 \pm 10.52$ 	& $0.8949$	& $36.6431 \pm 20.41$ \\
			\texttt{DR} $\colon s_{k} \equiv \frac{7}{5}$         	
			& $0.9976$	& $6.8541 \pm 6.61$ 	& $0.9952$	& $13.291 \pm 7.85$ 	& $0.9428$	& $\bf 33.5245 \pm 17.25$ \\
			\texttt{DR} $\colon s_{k} \equiv \frac{3}{2}$         		
			& $0.9980$	& $7.6111 \pm 6.54$ 	& $0.9957$	& $15.0055 \pm 7.06$ 	& $0.9514$	& $34.1237 \pm 15.8$ \\
			\texttt{DR} $\colon s_{k} \equiv \frac{7}{4}$         		
			& $0.9992$	& $12.1969 \pm 7.59$ 	& $0.9965$	& $27.535 \pm 5.69$ 	& $0.9644$	& $46.8346 \pm 10.99$ \\
			\texttt{DR} $\colon s_{k} = \frac{9}{5} - \frac{1}{k+2}$         	
			& $0.9992$	& $9.3479 \pm 6.96$ 	&  $\bf 0.9968$	& $23.3639 \pm 6.36$ 	& $0.9648$	& $47.6916 \pm 11.91$ \\
			\texttt{DR} $\colon s_{k} \equiv \frac{9}{5}$         	
			& $0.9994$	& $14.5185 \pm 8.64$ 	& $0.9966$	& $34.2205 \pm 6.08$ 	&  $\bf 0.9649$	&  $54.9974 \pm 10.09$ \\
			\texttt{DR} $\colon s_{k} = \frac{9}{5} + \frac{1}{k+2}$         	
			& $\bf 0.9996$	&  $23.1938 \pm 11.67$ 	& $0.9963$	& $47.9954 \pm 6.10$ 	& $0.9598$ 	& $64.4231 \pm 7.85$  \\
			\hline
			\texttt{Halpern}
			& $0.2400$	& $32.8750 \pm 22.87$ 	& $0.0000$ 	& $-//-$			& $0.0000$	& $-//-$ \\
			\hline
			\texttt{Fast KM} $\colon \alpha = 5$   									
			& $0.9690$	& $22.4536 \pm 18.26$	& $0.4973$	& $70.4548 \pm 16.42$	& $0.0000$	& $-//-$ \\
			\texttt{Fast KM} $\colon \alpha = 10$  									
			& $1.0000$	& $9.7566 \pm 6.83$ 	& $0.9996$	& $27.5816 \pm 11.78$ 	& $0.8753$	& $65.5713 \pm 14.46$ \\
			\texttt{Fast KM} $\colon \alpha = 30$  									
			& $1.0000$	& $4.9323 \pm 2.23$ 	& $1.0000$	& $10.0186 \pm 2.41$ 	& $1.0000$	& $17.6134 \pm 3.3$ \\
			\texttt{Fast KM} $\colon \alpha = 100$ 									
			& $1.0000$	& $3.5014 \pm 1.35$ 	& $1.0000$	& $6.2383 \pm 1.19$ 	& $1.0000$	& $9.5427 \pm 1.43$ \\
			\texttt{Fast KM} $\colon \alpha = 500$ 									
			& $1.0000$	& $2.6151 \pm 1.00$  	& $1.0000$	& $4.3118 \pm 0.73$ 	& $1.0000$	& $6.2944 \pm 0.75$ \\
			\bottomrule
		\end{tabular}
		\caption{The ratio of successfully solved problems and the average number of iterations for several algorithms}
		\label{tab:small}
	\end{center}
\end{table}

\cref{tab:large} shows the results obtained for the Fast KM algorithm for larger values of $n$, $\mathtt{Tol} := 10^{-8}$ and $\mathtt{k_{\max}} := 200$. It emphasizes once more that the numerical performances of our method become better when $\alpha$ takes larger values. Different from the class of DR algorithms, the growing dimension seems to less affect the number of iterations needed by Fast KM to provide a solution.

\begin{table}[!htb]
	\begin{center}
		\renewcommand{\arraystretch}{2}	
\tiny
		\begin{tabular}{l|l|c|l|c}
			\toprule
			{\bf $(n, \mathtt{N_{test}}, \mathtt{N_{init}})$} & \multicolumn{2}{|c|}{$(500; 100; 500)$} & \multicolumn{2}{|c}{$(5000,50,100)$} \\
			\hline
			\multicolumn{1}{c|}{\textbf{method}} & \textbf{ratio} & \textbf{iterations} & \textbf{ratio} & \textbf{iterations} \\
			\midrule
			\texttt{Fast KM}  $\colon \alpha = 10$ &								
			$0.6076$	& $154.5383 \pm 24.31$ 		& $0.0000$		& $-//-$ \\
			\texttt{Fast KM}  $\colon \alpha = 30$ &									
			$1.0000$ 		& $29.3096 \pm 5.14$ 		& $1.0000$		& $40.7248 \pm 3.56$ \\
			\texttt{Fast KM}  $\colon \alpha = 100$ &									
			$1.0000$ 		& $13.8564 \pm 1.77$ 		& $1.0000$		& $17.4264 \pm 1.18$ \\
			\texttt{Fast KM}  $\colon \alpha = 500$ &								
			$1.0000$ 		& $8.5773 \pm 0.94$ 		& $1.0000$		& $10.282 \pm 0.70$ \\
			\bottomrule
		\end{tabular}
		\caption{The ratio of successfully solved problems and the average number of iterations for the Fast KM algorithm}		
		\label{tab:large}
	\end{center}
\end{table}

To further illustrate the behaviour of the considered numerical methods, we plot below some generated trajectories in case $n:=1$,  and for $u = (1,5)^T$ and $\nu := 6$. The generated sequences by the different methods may converge to different solutions for the same starting point which is plotted as a black square in the figures. However, the way the iterates of the DR algorithms and the Halpern algorithm, on the one hand, an the Fast KM algorithm, on the other hand, tend to their limits through are totally different. While the iterates of the DR algorithms and the Halpern algorithm move along a curve above the hyperplane $H_{u, \nu}$, the ones generated by the Fast KM algorithm approach in a more straight manner the solution.

In the figures \cref{fig:projs2} and \cref{fig:projs2zoom} we plot the trajectory generated by the Halpern algorithm for two different starting points in order to emphasize their pronounced spiralling around the limit point. This also explains why (see also \cref{fig:projHalpern}), even if the algorithm finds a solution in less than $\mathtt{k_{\max}}$ steps, it requires more iterations than the DR algorithms and significantly more than the Fast KM algorithm.

\begin{figure}[!htb]
	\minipage{0.48\textwidth}
	\includegraphics[width=\linewidth]{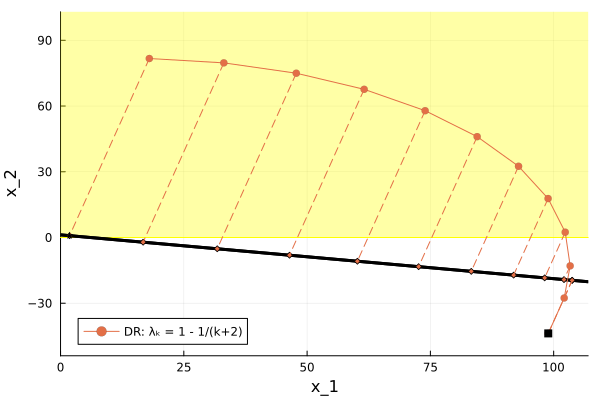}
	\caption{The trajectory of the \DR algorithm for $s_{k} := 1 - \frac{1}{k+2}$}
	\label{fig:projDR10mk}
	\endminipage\hfill
	\minipage{0.48\textwidth}
	\includegraphics[width=\linewidth]{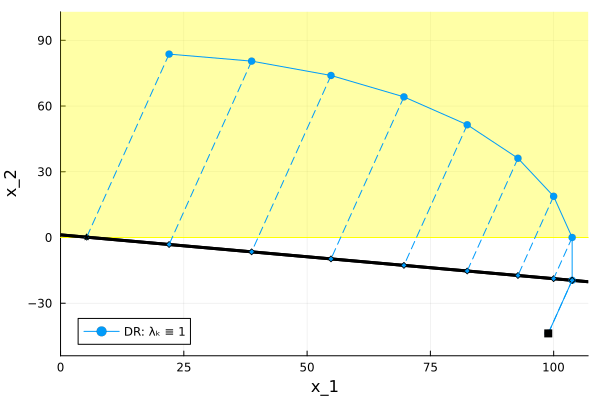}
	\caption{The trajectory of the \DR algorithm for $s_{k} := 1$}
	\label{fig:projDR10}
	\endminipage
\end{figure}

\begin{figure}[!htb]
	\minipage{0.48\textwidth}
	\includegraphics[width=\linewidth]{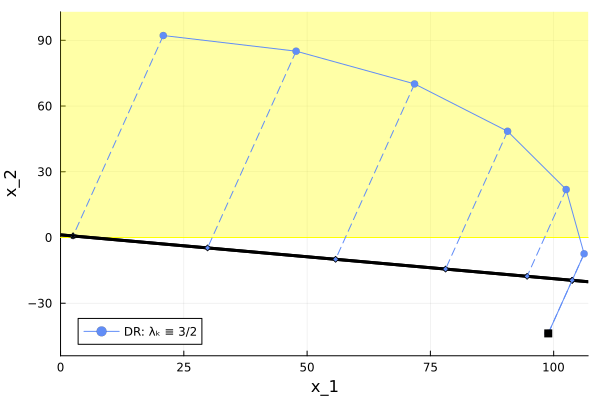}
	\caption{The trajectory of the \DR algorithm for $s_{k} := \frac{3}{2}$}
	\label{fig:projDR15}
	\endminipage\hfill
	\minipage{0.48\textwidth}
	\includegraphics[width=\linewidth]{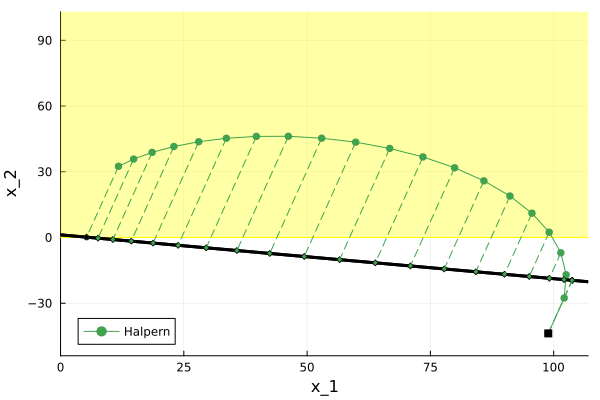}
	\caption{The trajectory of the Halpern algorithm}
	\label{fig:projHalpern}
	\endminipage
\end{figure}

\begin{figure}[!htb]
	\minipage{0.48\textwidth}
	\includegraphics[width=\linewidth]{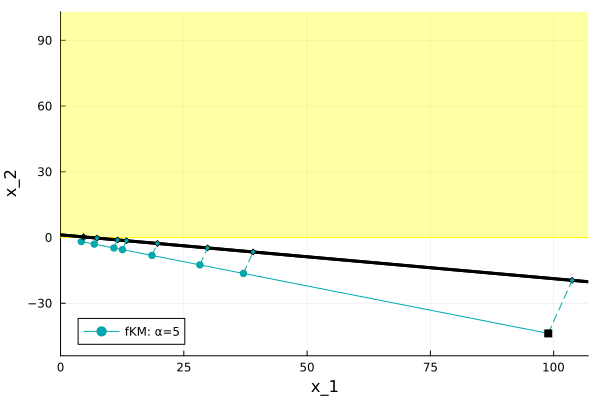}
	\caption{The trajectory of the Fast KM algorithm for $\alpha = 5$}
	\label{fig:projfKM5}
	\endminipage\hfill
	\minipage{0.48\textwidth}
	\includegraphics[width=\linewidth]{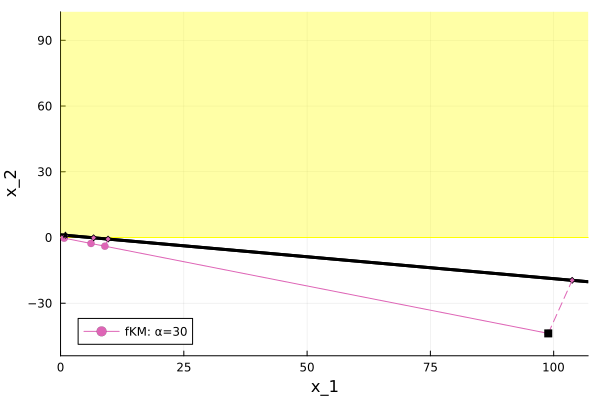}
	\caption{The trajectory of the Fast KM algorithm for $\alpha = 30$}
	\label{fig:projfKM30}
	\endminipage
\end{figure}

\begin{figure}[!htb]
	\minipage{0.48\textwidth}
	\includegraphics[width=\linewidth]{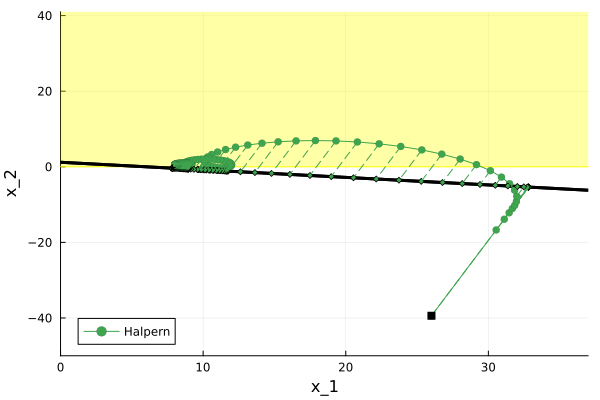}
	\caption{An instance for which the Halpern algorithm does not terminates in less than $\mathtt{k_{\max}}$ iterations}
	\label{fig:projs2}
	\endminipage\hfill
	\minipage{0.48\textwidth}
	\includegraphics[width=\linewidth]{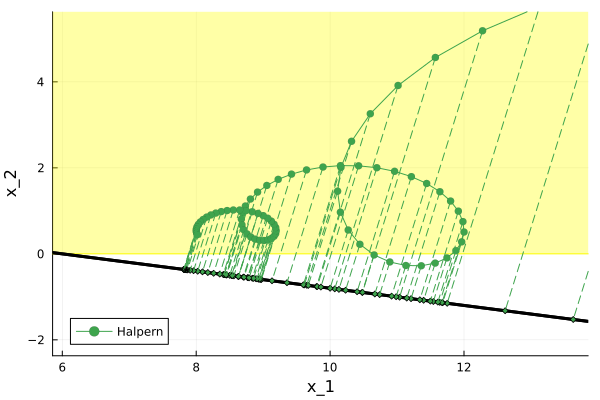}
	\caption{The spiral behavior of the Halpern algorithm}
	\label{fig:projs2zoom}
	\endminipage
\end{figure}

\appendix

\section*{Appendix}

In the appendix, we have compiled some auxiliary results and provided the proofs of the two technical lemmas  used in the convergence analysis of the Fast KM algorithm.

\section{Auxiliary results}
\label{appendix:aux}

The Opial Lemma (\cite{Opial}) is used in the proof of the convergence of the iterates.
\begin{lemma}
	\label{lem:Opial:dis}
	Let $\mathcal{S}$ be a nonempty subset of $\sH$ and $\left(x_{k} \right)_{k \geq 0}$ be a sequence in $\sH$.
	Assume that
	\begin{enumerate}
		\item
		\label{lem:Opial:dis:i}
		for every $x_{*} \in \mathcal{S}$, $\lim\limits_{k \to + \infty} \left\lVert x_{k} - x_{*} \right\rVert$ exists;
		
		\item
		\label{lem:Opial:dis:ii}
		every weak sequential cluster point of the sequence $\left(x_{k} \right)_{k \geq 0}$ as $k \to + \infty$ belongs to $\mathcal{S}$.
	\end{enumerate}
	Then $\left(x_{k} \right)_{k \geq 0}$ converges weakly to a point in $\mathcal{S}$ as $k \to + \infty$.
\end{lemma}

For the proof of the following result, which is the discrete counterpart of \cite[Lemma A.2]{Attouch-Peypouquet-Redont}, we refer to \cite[Lemma 21]{Bot-Csetnek-Nguyen}.

\begin{lemma}
	\label{lem:lim-u-k}
	Let $a \geq 1$ and $\left( q_{k} \right)_{k \geq 0}$ be a bounded sequence in $\sH$ such that
	\begin{equation*}
		\lim\limits_{k \to + \infty} \left( q_{k+1} + \dfrac{k}{a} \left( q_{k+1} - q_{k} \right) \right) = l \in \sH .
	\end{equation*}
	Then it holds $\lim\limits_{k \to + \infty} q_{k} = l$.
\end{lemma}

The following result is a particular instance of \cite[Lemma 5.31]{Bauschke-Combettes:book}.
\begin{lemma}
	\label{lem:quasi-Fej}
	Let $\left(a_{k} \right) _{k \geq 1}$, $\left( b_{k} \right)_{k \geq 1}$ and $\left(d_{k} \right)_{k \geq 1}$ be sequences of real numbers. Assume that $\left(a_{k} \right)_{k \geq 1}$ is bounded from below, and $\left( b_{k} \right)_{k \geq 1}$ and $\left( d_{k} \right)_{k \geq 1}$ are nonnegative  sequences such that $\sum_{k \geq 1} d_{k} < + \infty$. If
	\begin{equation*}
		a_{k+1} \leq a_{k} - b_{k} + d_{k} \quad \forall k \geq 1,
	\end{equation*}
	then the following statements are true:
	\begin{enumerate}
		\item the sequence $\left(b_{k} \right)_{k \geq 1}$ is summable, namely $\sum_{k \geq 1} b_{k} < + \infty$;
		\item the sequence $\left( a_{k} \right)_{k \geq 1}$ is convergent.
	\end{enumerate}
\end{lemma}

The following elementary result is used several times in the paper.
\begin{lemma}
	\label{lem:quad}
	Let $a, b, c \in \sR$ be such that $a < 0$ and $b^{2} - ac \leq 0$.
	Then it holds
	\begin{equation*}
		a \left\lVert x \right\rVert ^{2} + 2b \left\langle x , y \right\rangle + c \left\lVert y \right\rVert ^{2} \leq 0 \quad \forall x, y \in \sH .
	\end{equation*}
\end{lemma}

\pagebreak

\section{Proofs of the technical lemmas used in the analysis of the Fast KM algorithm}
\label{appendix:proof}

In this subsection we provide the proofs of \cref{lem:dec} and \cref{lem:trunc}.

\begin{proof}[Proof of \cref{lem:dec}]
	\labelcref{lem:dec:dif} Let $0 \leq \lambda \leq \alpha-1$. First we will show that for every $k \geq 1$ the following identity holds
	\begin{align}
		\label{dec:pre} 
		& \left( \E_{\lambda,k+1} + \dfrac{1}{4 \left( \alpha - 1 \right)} \left( \alpha - 2 \right) \alpha s^{2} \left( k+1 \right) \left\lVert \left( \Id - T \right) \left( x_{k} \right) \right\rVert ^{2} \right) \\
		& \quad \ \ - \left( \E_{\lambda,k} + \dfrac{1}{4 \left( \alpha - 1 \right)} \left( \alpha - 2 \right) \alpha s^{2} k \left\lVert \left( \Id - T \right) \left( x_{k-1} \right) \right\rVert ^{2} \right) \nonumber \\
		= & \ 2 \left( 2 - \alpha \right) \lambda s \left\langle x_{k+1} - x_{*} , \left( \Id - T \right) \left( x_{k} \right) \right\rangle 
		+ 2 \left( \lambda + 1 - \alpha \right) \left( 2k + \alpha + 1 \right) \left\lVert x_{k+1} - x_{k} \right\rVert ^{2} \nonumber \\
		& + \dfrac{1}{4 \left( \alpha - 1 \right)} \left( 2 - \alpha \right) s^{2} \left( 2 \left( 3 \alpha - 2 \right) k + 2 \alpha ^{2} + \alpha - 2 \right) \left\lVert \left( \Id - T \right) \left( x_{k} \right) \right\rVert ^{2} \nonumber \\
		& + \dfrac{1}{\alpha - 1} \Bigl( 4 \left( \alpha - 1 \right) \left( \lambda + 1 - \alpha \right) + \alpha \left( 2 - \alpha \right) \Bigr) sk \left\langle x_{k+1} - x_{k} , \left( \Id - T \right) \left( x_{k} \right) \right\rangle \nonumber \\
		& + \dfrac{1}{\alpha - 1} \left( 2 \alpha \left( \alpha - 1 \right) \left( \lambda + 1 - \alpha \right) + \alpha - 2 \left( \alpha - 1 \right) ^{2} \right) s \left\langle x_{k+1} - x_{k} , \left( \Id - T \right) \left( x_{k} \right) \right\rangle \nonumber \\
		& + \dfrac{1}{\alpha - 1} \left( 2 - \alpha \right) s \left( k + \alpha \right) k \left\langle x_{k+1} - x_{k} , \left( \Id - T \right) \left( x_{k} \right) - \left( \Id - T \right) \left( x_{k-1} \right) \right\rangle \nonumber \\
		& + \dfrac{1}{4 \left( \alpha - 1 \right)} \left( 2 - \alpha \right) s^{2} k \bigl( 2k + \alpha \bigr) \left\lVert \left( \Id - T \right) \left( x_{k} \right) - \left( \Id - T \right) \left( x_{k-1} \right) \right\rVert ^{2}. \nonumber
	\end{align}
	
	For brevity we denote for every $k \geq 0$
	\begin{align}
		\label{defi:u-k-lambda}
		u_{\lambda,k+1} & := 2 \lambda \left( x_{k+1} - x_{*} \right) + 2 \left( k + 1 \right) \left( x_{k+1} - x_{k} \right) \\
		& \qquad + \dfrac{1}{2 \left( \alpha - 1 \right)} \left( 3 \alpha - 2 \right) s \left( k + 1 \right) \left( \Id - T \right) \left( x_{k} \right), \nonumber 
	\end{align}
	which means that for every $k \geq 1$ it holds
	\begin{equation}
		\label{defi:u-k-lambda:pre}
		u_{\lambda,k} = 2 \lambda \left( x_{k} - x_{*} \right) + 2 k \left( x_{k} - x_{k-1} \right) + \dfrac{1}{2 \left( \alpha - 1 \right)} \left( 3 \alpha - 2 \right) sk \left( \Id - T \right) \left( x_{k-1} \right) .
	\end{equation}
	
	Subtracting \eqref{defi:u-k-lambda:pre} from \eqref{defi:u-k-lambda} and then using \eqref{dis:d-u} we obtain for every $k \geq 1$
	\begin{align}
		&  \ u_{\lambda,k+1} - u_{\lambda,k} \label{defi:u-k-lambda:dif} \\
		= & \ 2 \left( \lambda + 1 - \alpha \right) \left( x_{k+1} - x_{k} \right) + 2 \left( k + \alpha \right) \left( x_{k+1} - x_{k} \right) - 2k \left( x_{k} - x_{k-1} \right) \nonumber \\
		& + \dfrac{1}{2 \left( \alpha - 1 \right)} \left( 3 \alpha - 2 \right) s \left( \Id - T \right) \left( x_{k} \right) \nonumber \\
		& + \dfrac{1}{2 \left( \alpha - 1 \right)} \left( 3 \alpha - 2 \right) sk \Bigl( \left( \Id - T \right) \left( x_{k} \right) - \left( \Id - T \right) \left( x_{k-1} \right) \Bigr) \nonumber \\
		= & \ 2 \left( \lambda + 1 - \alpha \right) \left( x_{k+1} - x_{k} \right) + \dfrac{1}{2 \left( \alpha - 1 \right)} \left( \alpha - 2 \left( \alpha - 1 \right) ^{2} \right) s \left( \Id - T \right) \left( x_{k} \right) \nonumber \\
		& + \dfrac{1}{2 \left( \alpha - 1 \right)} \left( 2 - \alpha \right) sk \Bigl( \left( \Id - T \right) \left( x_{k} \right) - \left( \Id - T \right) \left( x_{k-1} \right) \Bigr) . \nonumber
	\end{align}
	In the following we want to use the identity
	\begin{equation}
		\label{dec:dif:u-lambda:pre}
		\dfrac{1}{2} \left( \left\lVert u_{\lambda,k+1} \right\rVert ^{2} - \left\lVert u_{\lambda,k} \right\rVert ^{2} \right) = \left\langle u_{\lambda,k+1} , u_{\lambda,k+1} - u_{\lambda,k} \right\rangle - \dfrac{1}{2} \left\lVert u_{\lambda,k+1} - u_{\lambda,k} \right\rVert ^{2} \quad \forall k \geq 1.
	\end{equation}
	Thanks to the relations \eqref{defi:u-k-lambda} and \eqref{defi:u-k-lambda:dif},  we derive for every $k \geq 1$ that
	\begin{align}
		\label{dec:dif:u-lambda:inn}
		& \left\langle u_{\lambda,k+1} , u_{\lambda,k+1} - u_{\lambda,k} \right\rangle \\
		= & \ 4 \lambda \left( \lambda + 1 - \alpha \right) \left\langle x_{k+1} - x_{*} , x_{k+1} - x_{k} \right\rangle \nonumber \\
		& + \dfrac{1}{\alpha - 1} \left( \alpha - 2 \left( \alpha - 1 \right) ^{2} \right) \lambda s \left\langle x_{k+1} - x_{*} , \left( \Id - T \right) \left( x_{k} \right) \right\rangle \nonumber \\
		& + \dfrac{1}{\alpha - 1} \left( 2 - \alpha \right) \lambda sk \left\langle x_{k+1} - x_{*} , \left( \Id - T \right) \left( x_{k} \right) - \left( \Id - T \right) \left( x_{k-1} \right) \right\rangle \nonumber \\
		& + 4 \left( \lambda + 1 - \alpha \right) \left( k + 1 \right) \left\lVert x_{k+1} - x_{k} \right\rVert ^{2} \nonumber \\
		& + \dfrac{1}{\alpha - 1} \Bigl( \alpha - 2 \left( \alpha - 1 \right) ^{2} + \left( 3 \alpha - 2 \right) \left( \lambda + 1 - \alpha \right) \Bigr) s \left( k + 1 \right) \left\langle x_{k+1} - x_{k} , \left( \Id - T \right) \left( x_{k} \right) \right\rangle \nonumber \\
		& + \dfrac{1}{\alpha - 1} \left( 2 - \alpha \right) s \left( k + 1 \right) k \left\langle x_{k+1} - x_{k} , \left( \Id - T \right) \left( x_{k} \right) - \left( \Id - T \right) \left( x_{k-1} \right) \right\rangle \nonumber \\
		& + \dfrac{1}{4 \left( \alpha - 1 \right) ^{2}} \left( 3 \alpha - 2 \right) \left( \alpha - 2 \left( \alpha - 1 \right) ^{2} \right) s^{2} \left( k + 1 \right) \left\lVert \left( \Id - T \right) \left( x_{k} \right) \right\rVert ^{2} \nonumber \\		
		& + \dfrac{1}{4 \left( \alpha - 1 \right) ^{2}} \left( 3 \alpha - 2 \right) \left( 2 - \alpha \right) s^{2} \left( k + 1 \right) k \left\langle \left( \Id - T \right) \left( x_{k} \right) , \right. \nonumber \\
		& \hspace{7.5cm} \left. \left( \Id - T \right) \left( x_{k} \right) - \left( \Id - T \right) \left( x_{k-1} \right) \right\rangle , \nonumber
	\end{align}
	and
	\begin{align}
		\label{dec:dif:u-lambda:norm}
		& - \dfrac{1}{2} \left\lVert u_{\lambda,k+1} - u_{\lambda,k} \right\rVert ^{2} \\
		= 	& - 2 \left( \lambda + 1 - \alpha \right) ^{2} \left\lVert x_{k+1} - x_{k} \right\rVert ^{2}
		- \dfrac{1}{8 \left( \alpha - 1 \right) ^{2}} \left( \alpha - 2 \left( \alpha - 1 \right) ^{2} \right) ^{2} s^{2} \left\lVert \left( \Id - T \right) \left( x_{k} \right) \right\rVert ^{2} \nonumber \\
		& - \dfrac{1}{8 \left( \alpha - 1 \right) ^{2}} \left( 2 - \alpha \right) ^{2} s^{2} k^{2} \left\lVert \left( \Id - T \right) \left( x_{k} \right) - \left( \Id - T \right) \left( x_{k-1} \right) \right\rVert ^{2} \nonumber \\
		& - \dfrac{1}{\alpha - 1} \left( \alpha - 2 \left( \alpha - 1 \right) ^{2} \right) \left( \lambda + 1 - \alpha \right) s \left\langle x_{k+1} - x_{k} , \left( \Id - T \right) \left( x_{k} \right) \right\rangle \nonumber \\
		& - \dfrac{1}{\alpha - 1} \left( 2 - \alpha \right) \left( \lambda + 1 - \alpha \right) sk \left\langle x_{k+1} - x_{k} , \left( \Id - T \right) \left( x_{k} \right) - \left( \Id - T \right) \left( x_{k-1} \right) \right\rangle \nonumber \\
		& - \dfrac{1}{4 \left( \alpha - 1 \right) ^{2}} \left( \alpha - 2 \left( \alpha - 1 \right) ^{2} \right) \left( 2 - \alpha \right) s^{2} k \left\langle \left( \Id - T \right) \left( x_{k} \right) , \right. \nonumber \\
		& \hspace{6cm} \left. \left( \Id - T \right) \left( x_{k} \right) - \left( \Id - T \right) \left( x_{k-1} \right) \right\rangle . \nonumber
	\end{align}
	A direct computation shows that for every $k \geq 0$
	\begin{align*}
		& \left( \left( 3 \alpha - 2 \right) \left( \lambda + 1 - \alpha \right) + \alpha - 2 \left( \alpha - 1 \right) ^{2} \right) \left( k + 1 \right) - \left( \lambda + 1 - \alpha \right) \left( \alpha - 2 \left( \alpha - 1 \right) ^{2} \right) \nonumber \\
		= 	& \left( \left( 3 \alpha - 2 \right) \left( \lambda + 1 - \alpha \right) + \alpha - 2 \left( \alpha - 1 \right) ^{2} \right) k  + 2 \alpha \left( \alpha - 1 \right) \left( \lambda + 1 - \alpha \right) + \alpha - 2 \left( \alpha - 1 \right) ^{2} \nonumber \\
		= 	& \left( \left( 3 \alpha - 2 \right) \left( \lambda + 1 - \alpha \right) + \alpha - 2 \left( \alpha - 1 \right) ^{2} - \left( 2 - \alpha \right) \lambda \right) k + \left( 2 - \alpha \right) \lambda k \nonumber \\
		& + 2 \alpha \left( \alpha - 1 \right) \left( \lambda + 1 - \alpha \right) + \alpha - 2 \left( \alpha - 1 \right) ^{2} \nonumber \\
		= 	& \Bigl( 4 \left( \alpha - 1 \right) \left( \lambda + 1 - \alpha \right) + \alpha \left( 2 - \alpha \right) \Bigr) k + \left( 2 - \alpha \right) \lambda k \nonumber \\
		& + 2 \alpha \left( \alpha - 1 \right) \left( \lambda + 1 - \alpha \right) + \alpha - 2 \left( \alpha - 1 \right) ^{2} .
	\end{align*}
	Therefore, by plugging \eqref{dec:dif:u-lambda:inn} and \eqref{dec:dif:u-lambda:norm} into \eqref{dec:dif:u-lambda:pre}, we get for every $k \geq 1$
	\vspace*{-8pt}
	\begin{align}
		\label{dec:dif:u-lambda}
		& \quad \dfrac{1}{2} \left( \left\lVert u_{\lambda,k+1} \right\rVert ^{2} - \left\lVert u_{\lambda,k} \right\rVert ^{2} \right) \\
		& = \ 4 \lambda \left( \lambda + 1 - \alpha \right) \left\langle x_{k+1} - x_{*} , x_{k+1} - x_{k} \right\rangle \nonumber \\
		& + \dfrac{1}{\alpha - 1} \left( \alpha - 2 \left( \alpha - 1 \right) ^{2} \right) \lambda s \left\langle x_{k+1} - x_{*} , \left( \Id - T \right) \left( x_{k} \right) \right\rangle \nonumber \\
		& + \dfrac{1}{\alpha - 1} \left( 2 - \alpha \right) \lambda sk \left\langle x_{k+1} - x_{*} , \left( \Id - T \right) \left( x_{k} \right) - \left( \Id - T \right) \left( x_{k-1} \right) \right\rangle \nonumber \\
		& + 2 \left( \lambda + 1 - \alpha \right) \left( 2k + \alpha + 1 - \lambda \right) \left\lVert x_{k+1} - x_{k} \right\rVert ^{2} \nonumber \\
		& + \dfrac{1}{8 \left( \alpha - 1 \right) ^{2}} \left( \alpha - 2 \left( \alpha - 1 \right) ^{2} \right) s^{2} \left( 2 \left( 3 \alpha - 2 \right) k + 2 \alpha ^{2} + \alpha - 2 \right) \left\lVert \left( \Id - T \right) \left( x_{k} \right) \right\rVert ^{2} \nonumber \\
		& + \dfrac{1}{\alpha - 1} \Bigl( 4 \left( \alpha - 1 \right) \left( \lambda + 1 - \alpha \right) + \alpha \left( 2 - \alpha \right) + \left( 2 - \alpha \right) \lambda \Bigr) sk \left\langle x_{k+1} - x_{k} , \left( \Id - T \right) \left( x_{k} \right) \right\rangle \nonumber \\
		& + \dfrac{1}{\alpha - 1} \left( 2 \alpha \left( \alpha - 1 \right) \left( \lambda + 1 - \alpha \right) + \alpha - 2 \left( \alpha - 1 \right) ^{2} \right) s \left\langle x_{k+1} - x_{k} , \left( \Id - T \right) \left( x_{k} \right) \right\rangle \nonumber \\
		& + \dfrac{1}{\alpha - 1} \left( 2 - \alpha \right) s \left( k + \alpha - \lambda \right) k \left\langle x_{k+1} - x_{k} , \left( \Id - T \right) \left( x_{k} \right) - \left( \Id - T \right) \left( x_{k-1} \right) \right\rangle \nonumber \\
		& + \dfrac{1}{4 \left( \alpha - 1 \right) ^{2}} \left( 2 - \alpha \right) s^{2} \left( \left( 3 \alpha - 2 \right) k + 2 \left( \alpha - 1 \right) \alpha \right) k \left\langle \left( \Id - T \right) \left( x_{k} \right) , \right. \nonumber \\
		& \hspace{6cm} \left. \left( \Id - T \right) \left( x_{k} \right) - \left( \Id - T \right) \left( x_{k-1} \right) \right\rangle \nonumber \\
		& - \dfrac{1}{8 \left( \alpha - 1 \right) ^{2}} \left( 2 - \alpha \right) ^{2} s^{2} k^{2} \left\lVert \left( \Id - T \right) \left( x_{k} \right) - \left( \Id - T \right) \left( x_{k-1} \right) \right\rVert ^{2} . \nonumber
	\end{align}
	
	Notice that by the definition of the energy function we have for every $k \geq 1$ 
	\vspace*{-5pt}
	\begin{align}\label{dec:energy+}
		& \ \E_{\lambda,k}  + \dfrac{1}{4 \left( \alpha - 1 \right)} \left( \alpha - 2 \right) \alpha s^{2} k \left\lVert \left( \Id - T \right) \left( x_{k-1} \right) \right\rVert ^{2}  \\
		= &  \ \dfrac{1}{2} \left\lVert u_{\lambda,k} \right\rVert ^{2} + 2 \lambda \left( \alpha - 1 - \lambda \right) \left\lVert x_{k} - x_{*} \right\rVert ^{2} \nonumber \\
		& + \dfrac{1}{\alpha - 1} \left( \alpha - 2 \right) \lambda sk \left\langle x_{k} - x_{*} , \left( \Id - T \right) \left( x_{k-1} \right) \right\rangle \nonumber \\
		& + \dfrac{1}{8 \left( \alpha - 1 \right)^2}(\alpha-2) s^2 k \left(\left( 3 \alpha - 2 \right) k + 2(\alpha-1) \alpha \right) \left\lVert \left( \Id - T \right) \left( x_{k-1} \right) \right\rVert ^{2}. \nonumber
	\end{align}
	Later we will subtract the above identity at consecutive indices and to this end we will make use of the following identities which hold for every $k \geq 1$:
	\begin{align}
		& \left\lVert x_{k+1} - x_{*} \right\rVert ^{2} - \left\lVert x_{k} - x_{*} \right\rVert ^{2} 
		=  2 \left\langle x_{k+1} - x_{*} , x_{k+1} - x_{k} \right\rangle - \left\lVert x_{k+1} - x_{k} \right\rVert ^{2}, \label{dec:dif:norm}
	\end{align}
	\begin{align}
		\label{dec:dif:vi}
		& \ 2 \lambda s \left( k + 1 \right) \left\langle x_{k+1} - x_{*} , \left( \Id - T \right) \left( x_{k} \right) \right\rangle - 2 \lambda sk \left\langle x_{k} - x_{*} , \left( \Id - T \right) \left( x_{k-1} \right) \right\rangle \\
		= & \ 2 \lambda sk \Bigl( \left\langle x_{k+1} - x_{*} , \left( \Id - T \right) \left( x_{k} \right) \right\rangle - \left\langle x_{k} - x_{*} , \left( \Id - T \right) \left( x_{k-1} \right) \right\rangle \Bigr) \nonumber \\
		& + 2 \lambda s \left\langle x_{k+1} - x_{*} , \left( \Id - T \right) \left( x_{k} \right) \right\rangle \nonumber \\
		= & \ 2 \lambda sk \left\langle x_{k+1} - x_{*} , \left( \Id - T \right) \left( x_{k} \right) - \left( \Id - T \right) \left( x_{k-1} \right) \right\rangle \nonumber \\
		& + 2 \lambda sk \left\langle x_{k+1} - x_{k} , \left( \Id - T \right) \left( x_{k-1} \right) \right\rangle + 2 \lambda s \left\langle x_{k+1} - x_{*} , \left( \Id - T \right) \left( x_{k} \right) \right\rangle  \nonumber \\
		= & \ 2 \lambda sk \left\langle x_{k+1} - x_{*} , \left( \Id - T \right) \left( x_{k} \right) - \left( \Id - T \right) \left( x_{k-1} \right) \right\rangle \nonumber \\
		& - 2 \lambda sk \left\langle x_{k+1} - x_{k} , \left( \Id - T \right) \left( x_{k} \right) - \left( \Id - T \right) \left( x_{k-1} \right) \right\rangle \nonumber \\
		& + 2 \lambda sk \left\langle x_{k+1} - x_{k} , \left( \Id - T \right) \left( x_{k} \right) \right\rangle + 2 \lambda s \left\langle x_{k+1} - x_{*} , \left( \Id - T \right) \left( x_{k} \right) \right\rangle, \nonumber
	\end{align}
	and
	\begin{align}
		\label{dec:dif:eq}
		& \ \dfrac{1}{4 \left( \alpha - 1 \right)} s^{2} \left( k + 1 \right) \bigl( \left( 3 \alpha - 2 \right) \left( k + 1 \right) + 2 \left( \alpha - 1 \right) \alpha \bigr) \left\lVert \left( \Id - T \right) \left( x_{k} \right) \right\rVert ^{2} \\
		& \ - \dfrac{1}{4 \left( \alpha - 1 \right)} s^{2} k \bigl( \left( 3 \alpha - 2 \right) k + 2 \left( \alpha - 1 \right) \alpha \bigr) \left\lVert \left( \Id - T \right) \left( x_{k-1} \right) \right\rVert ^{2} \nonumber \\
		=  & \ \dfrac{1}{4 \left( \alpha - 1 \right)} s^{2} \bigl( 2 \left( 3 \alpha - 2 \right) k + 2 \alpha ^{2} + \alpha - 2 \bigr) \left\lVert \left( \Id - T \right) \left( x_{k} \right) \right\rVert ^{2} \nonumber \\
		& \ + \dfrac{1}{4 \left( \alpha - 1 \right)} s^{2} k \left( \left( 3 \alpha - 2 \right) k + 2 \left( \alpha - 1 \right) \alpha \right) \left( \left\lVert \left( \Id - T \right) \left( x_{k} \right) \right\rVert ^{2} - \left\lVert \left( \Id - T \right) \left( x_{k-1} \right) \right\rVert ^{2} \right) \nonumber \\
		= & \ \dfrac{1}{4 \left( \alpha - 1 \right)} s^{2} \bigl( 2 \left( 3 \alpha - 2 \right) k + 2 \alpha ^{2} + \alpha - 2 \bigr) \left\lVert \left( \Id - T \right) \left( x_{k} \right) \right\rVert ^{2} \nonumber \\
		& \ + \dfrac{1}{2 \left( \alpha - 1 \right)} s^{2} k \bigl( \left( 3 \alpha - 2 \right) k + 2 \left( \alpha - 1 \right) \alpha \bigr) \left\langle \left( \Id - T \right) \left( x_{k} \right) , \right. \nonumber \\
		& \hspace{7cm} \left. \left( \Id - T \right) \left( x_{k} \right) - \left( \Id - T \right) \left( x_{k-1} \right) \right\rangle \nonumber \\
		& \ - \dfrac{1}{4 \left( \alpha - 1 \right)} s^{2} k \bigl( \left( 3 \alpha - 2 \right) k + 2 \left( \alpha - 1 \right) \alpha \bigr) \left\lVert \left( \Id - T \right) \left( x_{k} \right) - \left( \Id - T \right) \left( x_{k-1} \right) \right\rVert ^{2} . \nonumber
	\end{align}
	
	Therefore, by multiplying  \eqref{dec:dif:vi} and \eqref{dec:dif:eq} by $\frac{1}{2 \left( \alpha - 1 \right)} \left( \alpha - 2 \right)$,  relation \eqref{dec:energy+} gives for every $k \geq 1$
	\begin{align*}
		& \left( \E_{\lambda,k+1} + \dfrac{1}{4 \left( \alpha - 1 \right)} \left( \alpha - 2 \right) \alpha s^{2} \left( k+1 \right) \left\lVert \left( \Id - T \right) \left( x_{k} \right) \right\rVert ^{2} \right) \nonumber \\
		& \quad \ \ - \left( \E_{\lambda,k} + \dfrac{1}{4 \left( \alpha - 1 \right)} \left( \alpha - 2 \right) \alpha s^{2} k \left\lVert \left( \Id - T \right) \left( x_{k-1} \right) \right\rVert ^{2} \right) \nonumber \\
		& = \ \dfrac{1}{2} \left( \left\lVert u_{\lambda,k+1} \right\rVert ^{2} - \left\lVert u_{\lambda,k} \right\rVert ^{2} \right) + 2\lambda(\alpha - 1 - \lambda) \left(\left\lVert x_{k+1} - x_{*} \right\rVert ^{2} - \left\lVert x_{k} - x_{*} \right\rVert ^{2}  \right) \nonumber \\
		& + \dfrac{1}{\alpha - 1} \left( \alpha - 2 \right) \lambda s \left\langle x_{k+1} - x_{*} , \left( \Id - T \right) \left( x_{k} \right) \right\rangle \nonumber \\
		& + \dfrac{1}{\alpha - 1} \left( \alpha - 2 \right) \lambda sk \left\langle x_{k+1} - x_{*} , \left( \Id - T \right) \left( x_{k} \right) - \left( \Id - T \right) \left( x_{k-1} \right) \right\rangle \nonumber \\
		& - \dfrac{1}{\alpha - 1} \left( \alpha - 2 \right) \lambda sk \left\langle x_{k+1} - x_{k} , \left( \Id - T \right) \left( x_{k} \right) - \left( \Id - T \right) \left( x_{k-1} \right) \right\rangle \nonumber \\
		& + \dfrac{1}{\alpha - 1} \left( \alpha - 2 \right) \lambda sk \left\langle x_{k+1} - x_{k} , \left( \Id - T \right) \left( x_{k} \right) \right\rangle \nonumber \\
		& + \dfrac{1}{8 \left( \alpha - 1 \right) ^{2}} \left( \alpha - 2 \right) s^{2} \bigl( 2 \left( 3 \alpha - 2 \right) k + 2 \alpha ^{2} + \alpha - 2 \bigr) \left\lVert \left( \Id - T \right) \left( x_{k} \right) \right\rVert ^{2} \nonumber \\
		& + \dfrac{1}{4 \left( \alpha - 1 \right) ^{2}} \left( \alpha - 2 \right) s^{2} k \bigl( \left( 3 \alpha - 2 \right) k + 2 \left( \alpha - 1 \right) \alpha \bigr) \left\langle \left( \Id - T \right) \left( x_{k} \right) , \right. \nonumber \\
		& \hspace{7cm} \left. \left( \Id - T \right) \left( x_{k} \right) - \left( \Id - T \right) \left( x_{k-1} \right) \right\rangle \nonumber \\
		& - \dfrac{1}{8 \left( \alpha - 1 \right) ^{2}} \left( \alpha - 2 \right) s^{2} k \bigl( \left( 3 \alpha - 2 \right) k + 2 \left( \alpha - 1 \right) \alpha \bigr) \left\lVert \left( \Id - T \right) \left( x_{k} \right) - \left( \Id - T \right) \left( x_{k-1} \right) \right\rVert ^{2}
	\end{align*}
	By multiplying \eqref{dec:dif:norm} by $2\lambda(\alpha-1-\lambda)$ and by taking into consideration \eqref{dec:dif:u-lambda} and that
	\begin{align*}
		\alpha - 2 \left( \alpha - 1 \right) ^{2} + \alpha - 2 & = 2 \left( \alpha - 1 \right) \left( 2 - \alpha \right) \\ - \left( \alpha - 2 \right) ^{2} - \left( \alpha - 2 \right) \left( 3 \alpha - 2 \right) & = - 4 \left( \alpha - 2 \right) \left( \alpha - 1 \right),
	\end{align*}
	we immediately obtain from here identity \eqref{dec:pre}.
	
	Next we will focus on the term $\left\langle x_{k+1} - x_{k} , \left( \Id - T \right) \left( x_{k} \right) - \left( \Id - T \right) \left( x_{k-1} \right) \right\rangle$ for which we will provide an upper bound by exploiting the cocoercivity of $\Id - T$.
	Precisely, the relations \eqref{pre:coco} and \eqref{dis:d-u} guarantee that for every $k \geq 1$
	\begin{align*}
		& -2sk \left( k + \alpha \right) \left\langle x_{k+1} - x_{k} , \left( \Id - T \right) \left( x_{k} \right) - \left( \Id - T \right) \left( x_{k-1} \right) \right\rangle \nonumber \\
		= & -2sk^{2} \left\langle x_{k} - x_{k-1} , \left( \Id - T \right) \left( x_{k} \right) - \left( \Id - T \right) \left( x_{k-1} \right) \right\rangle \nonumber \\
		& + 2s^{2} k^{2} \left\lVert \left( \Id - T \right) \left( x_{k} \right) - \left( \Id - T \right) \left( x_{k-1} \right) \right\rVert ^{2} \nonumber \\
		& + \alpha s^{2} k \left\langle \left( \Id - T \right) \left( x_{k} \right) , \left( \Id - T \right) \left( x_{k} \right) - \left( \Id - T \right) \left( x_{k-1} \right) \right\rangle \nonumber \\
		\leq & \left( 2s - \dfrac{1}{\theta} \right) sk^{2} \left\lVert \left( \Id - T \right) \left( x_{k} \right) - \left( \Id - T \right) \left( x_{k-1} \right) \right\rVert ^{2}
		+ \dfrac{1}{2} \alpha s^{2} k \left\lVert \left( \Id - T \right) \left( x_{k} \right) \right\rVert ^{2} \nonumber \\
		& + \dfrac{1}{2} \alpha s^{2} k \left\lVert \left( \Id - T \right) \left( x_{k} \right) - \left( \Id - T \right) \left( x_{k-1} \right) \right\rVert ^{2}
		- \dfrac{1}{2} \alpha s^{2} k \left\lVert \left( \Id - T \right) \left( x_{k-1} \right) \right\rVert ^{2} \nonumber \\
		\leq & \left( \left( 2s - \dfrac{1}{\theta} \right) sk^{2} + \dfrac{1}{2} \alpha s^{2} k \right) \left\lVert \left( \Id - T \right) \left( x_{k} \right) - \left( \Id - T \right) \left( x_{k-1} \right) \right\rVert ^{2} \nonumber \\
		& + \dfrac{1}{2} \alpha s^{2} \left( k+1 \right) \left\lVert \left( \Id - T \right) \left( x_{k} \right) \right\rVert ^{2} - \dfrac{1}{2} \alpha s^{2} k \left\lVert \left( \Id - T \right) \left( x_{k-1} \right) \right\rVert ^{2} .
	\end{align*}
	After multiplying this inequality by $\frac{1}{2 \left( \alpha - 1 \right)} \left( \alpha - 2 \right) > 0$, adding it to\eqref{dec:pre}, and using that
	\begin{align*}
		& 2 \left( 2 - \alpha \right) \lambda s \left\langle x_{k+1} - x_{*} , \left( \Id - T \right) \left( x_{k} \right) \right\rangle \nonumber \\
		= \ 	& 2 \left( 2 - \alpha \right) \lambda s \left\langle x_{k+1} - x_{k} , \left( \Id - T \right) \left( x_{k} \right) \right\rangle + 2 \left( 2 - \alpha \right) \lambda s \left\langle x_{k} - x_{*} , \left( \Id - T \right) \left( x_{k} \right) \right\rangle ,
	\end{align*}
	we deduce the desired inequality \eqref{dec:inq}. To obtain the coefficients of $\left\lVert x_{k+1} - x_{k} \right\rVert ^{2}$ and $\left\lVert \left( \Id - T \right) \left( x_{k} \right) \right\rVert ^{2}$ as given in \eqref{dec:const}, one also has to take into consideration that $\lambda + 1 - \alpha \leq 0$ and $2 \alpha^{2} + \alpha - 2 > 0$, as $\alpha >2$. The assumptions we made on $\alpha$ and $\lambda$ immediately imply that $\omega_1$, $\omega_2$ and $\omega_4$ are nonpositive numbers.
	
	\labelcref{lem:dec:bnd}
	Since
	\begin{align*}
		& \dfrac{1}{\alpha - 1} \left( \alpha - 2 \right) \lambda sk \left\langle x_{k} - x_{*} , \left( \Id - T \right) \left( x_{k} \right) \right\rangle \nonumber \\
		& + \dfrac{1}{8 \left( \alpha - 1 \right) ^{2}} \left( \alpha - 2 \right) \left( 3 \alpha - 2 \right) s^{2} k^{2} \left\lVert \left( \Id - T \right) \left( x_{k} \right) \right\rVert ^{2} \nonumber \\
		= \ 	& \dfrac{1}{3 \alpha - 2} \left( \alpha - 2 \right) \left( \dfrac{1}{\alpha - 1} \left( 3 \alpha - 2 \right) \lambda sk \left\langle x_{k} - x_{*} , \left( \Id - T \right) \left( x_{k} \right) \right\rangle \right. \nonumber \\
		& \hspace{5cm} \left. + \dfrac{1}{8 \left( \alpha - 1 \right) ^{2}} \left( 3 \alpha - 2 \right) ^{2} s^{2} k^{2} \left\lVert \left( \Id - T \right) \left( x_{k} \right) \right\rVert ^{2} \right) \nonumber \\
		= \ 	& \dfrac{1}{3 \alpha - 2} \left( \alpha - 2 \right) \bigg( \dfrac{1}{2} \left\lVert 2 \lambda \left( x_{k} - x_{*} \right) + \dfrac{1}{2 \left( \alpha - 1 \right)} \left( 3 \alpha - 2 \right) sk \left( \Id - T \right) \left( x_{k-1} \right) \right\rVert ^{2} \nonumber \\
		& \hspace{9cm} - 2 \lambda^{2} \left\lVert x_{k} - x_{*} \right\rVert ^{2} \bigg),
	\end{align*}
	we deduce that for every $k \geq 1$
	\begin{align*}
		\E_{\lambda,k}
		& = \dfrac{1}{2} \left\lVert 2 \lambda \left( x_{k} - x_{*} \right) + 2k \left( x_{k} - x_{k-1} \right) + \dfrac{1}{2 \left( \alpha - 1 \right)} \left( 3 \alpha - 2 \right) sk\left( \Id - T \right) \left( x_{k-1} \right) \right\rVert ^{2} \nonumber \\
		& \quad + 2 \lambda \left( \alpha - 1 - \lambda \right) \left\lVert x_{k} - x_{*} \right\rVert ^{2} + \dfrac{1}{\alpha - 1} \left( \alpha - 2 \right) \lambda sk \left\langle x_{k} - x_{*} , \left( \Id - T \right) \left( x_{k-1} \right) \right\rangle \nonumber \\
		& \quad + \dfrac{1}{8 \left( \alpha - 1 \right) ^{2}} \left( \alpha - 2 \right) \left( 3 \alpha - 2 \right) s^{2} k^{2} \left\lVert \left( \Id - T \right) \left( x_{k-1} \right) \right\rVert ^{2} \nonumber \\
		& = \dfrac{1}{2} \left\lVert 2 \lambda \left( x_{k} - x_{*} \right) + 2k \left( x_{k} - x_{k-1} \right) + \dfrac{1}{2 \left( \alpha - 1 \right)} \left( 3 \alpha - 2 \right) sk\left( \Id - T \right) \left( x_{k-1} \right) \right\rVert ^{2} \nonumber \\
		& \quad + \dfrac{1}{2 \left( 3 \alpha - 2 \right)} \left( \alpha - 2 \right) \left\lVert 2 \lambda \left( x_{k} - x_{*} \right) + \dfrac{1}{2 \left( \alpha - 1 \right)} \left( 3 \alpha - 2 \right) sk \left( \Id - T \right) \left( x_{k-1} \right) \right\rVert ^{2} \nonumber \\
		& \quad + 2 \lambda \left( \alpha - 1 \right) \left( 1 - \dfrac{4 \lambda}{3 \alpha - 2} \right) \left\lVert x_{k} - x_{*} \right\rVert ^{2} .
	\end{align*}
	Using the identity 
	\begin{equation*}
		\left\lVert x \right\rVert ^{2} + \left\lVert y \right\rVert ^{2} = \frac{1}{2} \left( \left\lVert x + y \right\rVert ^{2} + \left\lVert x - y \right\rVert ^{2} \right) \quad \forall x, y \in \sH ,
	\end{equation*}
	we obtain for every $k \geq 1$
	\begin{align}
		\label{dec:eq}
		& \E_{\lambda,k} \\
		= \ & \dfrac{\alpha}{3 \alpha - 2} \left\lVert 2 \lambda \left( x_{k} - x_{*} \right) + 2k \left( x_{k} - x_{k-1} \right) + \dfrac{1}{2 \left( \alpha - 1 \right)} \left( 3 \alpha - 2 \right) sk \left( \Id - T \right) \left( x_{k-1} \right) \right\rVert ^{2} \nonumber \\
		& \quad + \dfrac{1}{4 \left( 3 \alpha - 2 \right)} \left( \alpha - 2 \right) \bigg\Vert 4 \lambda \left( x_{k} - x_{*} \right) + 2k \left( x_{k} - x_{k-1} \right)  \nonumber \\
		& \hspace{5cm}  + \dfrac{1}{\alpha - 1} \left( 3 \alpha - 2 \right) sk \left( \Id - T \right) \left( x_{k-1} \right) \bigg\Vert ^{2} \nonumber \\
		& \quad + \dfrac{1}{3 \alpha - 2} \left( \alpha - 2 \right) k^{2} \left\lVert x_{k} - x_{k-1} \right\rVert ^{2}
		+ 2 \lambda \left( \alpha - 1 \right) \left( 1 - \dfrac{4 \lambda}{3 \alpha - 2} \right) \left\lVert x_{k} - x_{*} \right\rVert ^{2} , \nonumber 
	\end{align}
	This shows that for $0 \leq \lambda \leq \frac{3 \alpha}{4} - \frac{1}{2}$ all terms in the expression \eqref{dec:eq} are nonnegative, thus the sequence $(\E_{\lambda,k})_{k \geq 1}$  is nonnegative, too. 
\end{proof}

\begin{proof}[Proof of \cref{lem:trunc}]
	For the quadratic expression in $R_{k}$ we calculate
	\begin{equation*}
		\dfrac{\Delta_{k}}{s^{2}}
		:= \left( \omega_{2} k + \omega_{3} \right) ^{2} - \dfrac{2(5\alpha-2)}{3\alpha -2} \omega_{1} \omega_{4} k^{2} \nonumber \\
		= \left( \omega_{2}^{2} - \dfrac{2(5\alpha-2)}{3\alpha -2}  \omega_{1} \omega_{4} \right) k^{2} + 2 \omega_{2} \omega_{3} k + \omega_{3}^{2} .
	\end{equation*}
	It suffices to guarantee that $\omega_{2}^{2} - \frac{2(5\alpha-2)}{3\alpha -2}  \omega_{1} \omega_{4} < 0$ in order to be sure that there exits some integer $k \left( \lambda \right) \geq 1$ such that $\Delta_{k} \leq 0$ for every $k \geq k \left( \lambda \right)$ and to obtain from here, due to \cref{lem:quad}, that $R_{k} \leq 0$ for every $k \geq k \left( \lambda \right)$.
	
	We will show that there exists a nonempty open interval contained in $[0, \alpha-1]$ with the property that $\omega_{2}^{2} - \dfrac{2(5\alpha-2)}{3\alpha -2}  \omega_{1} \omega_{4} < 0$ holds when $\lambda$ is chosen within this open interval.  To this end we set $\xi := \lambda + 1 - \alpha \leq 0$ and get
	\begin{equation*}
		\omega_{2} = \dfrac{1}{\alpha - 1} \Bigl( 4 \left( \alpha - 1 \right) \xi - \alpha \left( \alpha - 2 \right) \Bigr) 
		\quad \textrm{ and } \quad
		\omega_{1} \omega_{4} = - \dfrac{2}{\alpha - 1} \left( \alpha - 2 \right) \left( 3 \alpha - 2 \right) \xi .
	\end{equation*}
	Written in terms of $\xi$, we have first to guarantee that
	\begin{align}
		\label{trunc:omega-a}
		& \omega_{2}^{2} - \dfrac{2(5\alpha-2)}{3\alpha -2}  \omega_{1} \omega_{4} \nonumber \\
		= \ 	& \dfrac{1}{\left( \alpha - 1 \right) ^{2}} \left( \left( 4 \left( \alpha - 1 \right) \xi - \alpha \left(\alpha -2 \right) \right) ^{2} + 4 \left( 5 \alpha - 2 \right) \left( \alpha - 1 \right) \left( \alpha - 2 \right) \xi \right) \nonumber \\
		= \ 	& \dfrac{1}{\left( \alpha - 1 \right) ^{2}} \left( 16 \left( \alpha - 1 \right) ^{2} \xi^{2} + 4 \left( \alpha - 1 \right) \left( \alpha - 2 \right) \left( 3 \alpha - 2 \right) \xi + \alpha ^{2} \left( \alpha - 2 \right) ^{2} \right) < 0 . \nonumber
	\end{align}
	A direct computation shows that
	\begin{equation*}
		\Delta_{\xi}
		= 16 \left( \alpha - 1 \right) ^{2} \left( 2 - \alpha \right) ^{2} \left( \left( 3 \alpha - 2 \right) ^{2} - 4 \alpha ^{2} \right)
		= 16 \left( \alpha - 1 \right) ^{2} \left( \alpha - 2 \right) ^{3} \left( 5 \alpha - 2 \right) > 0.
	\end{equation*}
	Hence, in order to get \eqref{trunc:omega-a}, we have to choose $\xi$ between the two roots of the quadratic function arising in this formula, in other words
	\begin{align*}
		\xi_{1} \left( \alpha \right) & := \dfrac{1}{32 \left( \alpha - 1 \right) ^{2}} \left( - 4 \left( \alpha - 1 \right) \left( \alpha - 2 \right) \left( 3 \alpha - 2 \right) - \sqrt{\Delta_{\xi}} \right) \nonumber \\
		& = - \dfrac{1}{8 \left( \alpha - 1 \right)} \left( \alpha - 2 \right) \left( 3 \alpha - 2  + \sqrt{\left( \alpha - 2 \right) \left( 5 \alpha - 2 \right)} \right) \nonumber \\
		& < \xi = \lambda + 1 - \alpha
		< \xi_{2} \left( \alpha \right) := \dfrac{1}{32 \left( \alpha - 1 \right) ^{2}} \left( - 4 \left( \alpha - 1 \right) \left( \alpha - 2 \right) \left( 3 \alpha - 2 \right) + \sqrt{\Delta_{\xi}} \right) \nonumber \\
		& = - \dfrac{1}{8 \left( \alpha - 1 \right)} \left( \alpha - 2 \right) \left(3 \alpha - 2 - \sqrt{\left( \alpha - 2 \right) \left( 5 \alpha - 2 \right)} \right) .
	\end{align*}
	Obviously $\xi_{1} \left( \alpha \right) < 0$ and from Viète’s formula $\xi_{1} \left( \alpha \right) \cdot \xi_{2} \left( \alpha \right) = \frac{\alpha^{2} \left(\alpha -2\right) ^{2}}{16 \left( \alpha - 1 \right) ^{2}}$, it follows that we must have $\xi_{2} \left( \alpha \right) < 0$ as well. 
	
	Therefore, going back to $\lambda$, in order to be sure that $\omega_{2}^{2} - \frac{2(5\alpha-2)}{3\alpha -2}  \omega_{1} \omega_{4} < 0$ this must be chosen such that
	\begin{equation*}
		\alpha - 1 + \xi_{1} \left( \alpha \right) < \lambda < \alpha - 1 + \xi_{2} \left( \alpha \right) .
	\end{equation*}
	Next we will show that
	\begin{equation}
		\label{trunc:check}
		0 < \alpha - 1 - \dfrac{1}{8 \left( \alpha - 1 \right)} \left( \alpha - 2 \right) \left( 3 \alpha - 2 \right) < \dfrac{3 \alpha}{4} - \dfrac{1}{2}.
	\end{equation}
	Indeed, the inequality on the left-hand side follows immediately, since
	\begin{align*}
		\alpha - 1 - \dfrac{1}{8 \left( \alpha - 1 \right)} \left( \alpha - 2 \right) \left( 3 \alpha - 2 \right)
		& = \dfrac{1}{8 \left( \alpha - 1 \right)} \left( 5 \alpha^{2} - 8 \alpha + 4 \right) \nonumber \\ 
		& = \dfrac{1}{8 \left( \alpha - 1 \right)} \left( \alpha^{2} + 4 \left( \alpha - 1 \right) ^{2} \right) > 0 .
	\end{align*}
	Using this relation, one can notice that the inequality on the right-hand side of \eqref{trunc:check} can be equivalently written as
	\begin{equation*}
		5 \alpha^{2} - 8 \alpha + 4 < 2 \left( \alpha - 1 \right) \left( 3 \alpha - 2 \right) \Leftrightarrow 0 < \alpha^{2} - 2 \alpha = \alpha \left( \alpha - 2 \right),
	\end{equation*}
	which is true as $\alpha >2$.
	
	From \eqref{trunc:check} we immediately deduce that
	\begin{equation*}
		0 < \alpha - 1 + \xi_{2} \left( \alpha \right)
		\quad \textrm{ and } \quad
		\alpha - 1 + \xi_{1} \left( \alpha \right) < \dfrac{3 \alpha}{4} - \dfrac{1}{2}.
	\end{equation*}
	This allows us to choose
	\begin{align*}
		\underline{\lambda} \left( \alpha \right) & := \alpha - 1 + \xi_{1} \left( \alpha \right) \nonumber \\ 
		& = \dfrac{\alpha^{2}}{8 \left( \alpha - 1 \right)} + \dfrac{\alpha - 1}{2} - \dfrac{1}{8 \left( \alpha - 1 \right)} \left( \alpha - 2 \right) \sqrt{\left( \alpha - 2 \right) \left( 5 \alpha - 2 \right)}  \nonumber \\
		< \overline{\lambda} \left( \alpha \right) & := \min \left\lbrace \dfrac{3 \alpha}{4} - \dfrac{1}{2},   \alpha - 1 + \xi_{2} \left( \alpha \right) \right\rbrace  \\
		& = \min \left\lbrace \dfrac{3 \alpha}{4} - \dfrac{1}{2} , \dfrac{\alpha^{2}}{8 \left( \alpha - 1 \right)} + \dfrac{\alpha - 1}{2} + \dfrac{1}{8 \left( \alpha - 1 \right)} \left( \alpha - 2 \right) \sqrt{\left( \alpha - 2 \right) \left( 5 \alpha - 2 \right)} \right\rbrace ,
	\end{align*}
	since 
	\begin{equation*}
		\dfrac{1}{8 \left( \alpha - 1 \right)} \alpha^{2} + \dfrac{1}{2} \left( \alpha - 1 \right) - \dfrac{1}{8 \left( \alpha - 1 \right)} \left( \alpha - 2 \right) \sqrt{\left( \alpha - 2 \right) \left( 5 \alpha - 2 \right)} > 0 .
	\end{equation*}
	Indeed, as $\left( \alpha - 1 \right) \sqrt{\alpha - 1} > \left( \alpha - 2 \right) \sqrt{\alpha - 2}$ and $4 \sqrt{\alpha - 1} > \sqrt{5 \alpha - 2}$ we can easily deduce that
	\begin{equation*}
		\alpha^{2} + 4 \left( \alpha - 1 \right) ^{2} > 4 \left( \alpha - 1 \right) ^{2} > \left( \alpha - 2 \right) \sqrt{\left( \alpha - 2 \right) \left( 5 \alpha - 2 \right)}
	\end{equation*}
	and the claim follows.
	
	In conclusion, choosing $\lambda$ to satisfy $\underline{\lambda} \left( \alpha \right) < \lambda < \overline{\lambda} \left( \alpha \right)$, we have $$\omega_{2}^{2} - \frac{2(5\alpha-2)}{3\alpha -2}  \omega_{1} \omega_{4} < 0$$ and therefore there exists some integer $k \left( \lambda \right) \geq 1$ such that $R_{k} \leq 0$ for every $k \geq k \left( \lambda \right)$.
\end{proof}

\section*{Acknowledgments}
The authors express their gratitude to the handling editor and three anonymous reviewers for their valuable comments and remarks, which significantly improved the quality of the manuscript.

\bibliographystyle{siamplain}
\bibliography{references}
\end{document}